\documentclass[12pt]{amsart}
\pdfoutput=1
\usepackage{amssymb}
\usepackage{cite}
\usepackage{array}
\usepackage{booktabs}
\usepackage{mdwtab}
\usepackage{mathtools}
\usepackage[T1]{fontenc}
\usepackage[utf8]{inputenc}
\usepackage{hyphenat}
\usepackage{enumitem}
\usepackage{ifpdf}
\usepackage{braket}
\usepackage{mathabx}
\usepackage{mathrsfs}
\usepackage{color}
\usepackage{minibox} 
\ifpdf
  \usepackage[pdftex]{graphicx}
  \usepackage[pdftex,rmargin=1in,lmargin=1in]{geometry}
  \usepackage[bookmarks=true, bookmarksopen=true,%
    bookmarksdepth=3,bookmarksopenlevel=2,%
    colorlinks=true,%
    linkcolor=dark-blue,%
    citecolor=dark-blue,%
    filecolor=dark-blue,%
    menucolor=dark-blue,%
    urlcolor=dark-blue]{hyperref}
  \hypersetup{pdftitle={Conformal surface embeddings and extremal length}}
  \hypersetup{pdfauthor={Jeremy Kahn, Kevin M. Pilgrim, and Dylan P. Thurston}}
\else
  \usepackage[dvips]{graphicx}
  \usepackage[dvips,margin=1in]{geometry}
  \usepackage[draft]{hyperref}
\fi
\usepackage{url}

\usepackage{tikz}
\usetikzlibrary{arrows}
\tikzstyle{every picture}=[> = to]
\tikzset{cdlabel/.style={execute at begin node=$\scriptstyle,execute at end node=$}}
\tikzset{implication/.style={double equal sign distance, -implies}}
\tikzset{biimplication/.style={double equal sign distance, implies-implies}}

\makeatletter
\newcommand\mi@kern[1]{%
  \settowidth\@tempdima{$\mi@obj^{#1}$}
  \kern-\@tempdima
  #1
  \settowidth\@tempdima{$\mi@obj$}
  \kern\@tempdima
}

\newtoks\mi@toksp
\newtoks\mi@toksb
\DeclareRobustCommand{\manyindices}[5]{
  \def\mi@obj{#5}
  \mi@toksp\expandafter{\mi@kern{#2}}
  \mi@toksb\expandafter{\mi@kern{#1}}
  \@mathmeasure4\textstyle{#5_{#1}^{#2}}
  \@mathmeasure6\textstyle{#5_{#3}^{#4}}
  \dimen0-\wd6 \advance\dimen0\wd4
  \@mathmeasure8\textstyle{\hphantom{{}_{#1}^{#2}}#5^{\the\mi@toksp#4}_{\the\mi@toksb#3}}
  \hbox to \dimen0{}{\kern-\dimen0\box8}
}
\makeatother 

\newread\testin

\def\mathcenter#1{\vcenter{\hbox{$#1$}}}

\def\graphb#1{\includegraphics[trim=-1 -1 -1 -1]{#1}}

\def\mfigb#1{\mathcenter{\graphb{#1}}}

\renewcommand{\colon}{\nobreak\mskip2mu\mathpunct{}\nonscript
  \mkern-\thinmuskip{:}\allowbreak\mskip6muplus1mu\relax}

\ifpdf
  
\else
  
\fi

\newcommand{\RR}{\mathbb R}
\newcommand{\CC}{\mathbb C}
\newcommand{\DD}{\mathbb D}
\newcommand{\ZZ}{\mathbb Z}
\newcommand{\QQ}{\mathbb Q}

\newcommand{\co}{\colon}
\newcommand{\eps}{\varepsilon}
\renewcommand{\epsilon}{\varepsilon}
\newcommand{\abs}[1]{\lvert #1 \rvert}
\newcommand{\norm}[1]{\lVert #1 \rVert}

\newcommand{\bdy}{\partial}

\renewcommand{\Im}{\mathop{\mathrm{Im}}}
\renewcommand{\Re}{\mathop{\mathrm{Re}}}

\theoremstyle{plain}
\numberwithin{equation}{section}
\newtheorem{proposition}[equation]{Proposition}
\newtheorem{lemma}[equation]{Lemma}

\newtheorem{conjecture}[equation]{Conjecture}

\newtheorem{theorem}{Theorem}
\newtheorem{citethm}[equation]{Theorem}

\theoremstyle{definition}
\newtheorem{definition}[equation]{Definition}
\newtheorem{convention}[equation]{Convention}

\newtheorem{problem}[equation]{Problem}
\newtheorem{claim}[equation]{Claim}

\theoremstyle{remark}
\newtheorem{example}[equation]{Example}
\newtheorem{remark}[equation]{Remark}

\theoremstyle{plain}

\newcommand{\closure}[1]{\overline{#1}}

\newcommand{\D}{\DD}

\DeclareMathOperator{\EL}{EL} 
\DeclareMathOperator{\SF}{SF} 

\DeclareMathOperator{\Mod}{Mod} 
\newcommand{\Teich}{\mathcal{T}}

\newcommand{\CCa}{\widehat{\CC}}
\newcommand{\MF}{\mathcal{MF}}
\newcommand{\Curves}{\mathcal{C}}
\newcommand{\Quad}{\mathcal{Q}}

\newcommand{\Meas}{\mathcal{M}}

\newcommand{\id}{\mathrm{id}}

\newcommand{\wt}[1]{\widetilde{#1}}

\newcommand{\oJ}{{\overline{J}}}
\newcommand{\oS}{{\overline{S}}}
\newcommand{\oDD}{\overline{\DD}}
\newcommand{\Fh}{\mathcal{F}_h}
\newcommand{\Fv}{\mathcal{F}_v}

\definecolor{dark-green}{rgb}{0,0.6,0}
\definecolor{dark-red}{rgb}{0.7,0,0}
\definecolor{dark-blue}{rgb}{0,0,0.8}

\graphicspath{{draws/}}

\begin{document}
\title{Conformal surface embeddings and extremal length}

\author[Kahn]{Jeremy Kahn}
\address{Brown University\\
	 151 Thayer Street,
         Providence, RI 02912\\
         USA}
\email{jeremy\_kahn@brown.edu}

\author[Pilgrim]{Kevin M.~Pilgrim}
\address{Indiana University\\
         831 E. Third St.,
         Bloomington, Indiana 47405\\
         USA}
\email{pilgrim@indiana.edu}

\author[Thurston]{Dylan~P.~Thurston}
\address{Indiana University\\
         831 E. Third St.,
         Bloomington, Indiana 47405\\
         USA}
\email{dpthurst@indiana.edu}

\subjclass[2010]{Primary 30F60; Secondary 31A15, 32G15}
\keywords{Riemann surfaces with boundary, conformal embeddings, extremal length}

\begin{abstract}
  Given two Riemann surfaces with boundary and a homotopy class of
  topological embeddings between them, there is a conformal embedding
  in the homotopy class if and only if the extremal length of every simple
  multi-curve is decreased under the embedding. Furthermore, the
  homotopy class has a conformal embedding that misses an open disk
  if and only if extremal lengths are decreased by a definite
  ratio. This ratio remains bounded away from one under finite covers.
\end{abstract}

\maketitle

  \tableofcontents

\section{Introduction}
\label{sec:intro}

Let $R$ and $S$ be two Riemann surfaces of finite topological type,
possibly with
boundary, and let $f \co
R \hookrightarrow S$ be a topological embedding. The goal
of this paper is to give conditions for $f$ to
be homotopic to a \emph{conformal} embedding, possibly with extra nice
properties.
We give an answer in terms of ratios of extremal
lengths of simple multi-curves.

For us, surfaces~$S$ are of finite type and a simple multi-curve
on~$S$ is an embedded 1-manifold
in~$S$. See Definitions~\ref{def:surface}
and~\ref{def:curves} for the full
definitions.
The \emph{extremal length} $\EL_S[C]$ of a simple curve~$C$
is a measure of the fattest annulus
that can be embedded in~$S$ with core curve isotopic to~$C$. See
Section~\ref{sec:ext-length} for more on extremal length of
multi-curves.

\begin{definition}\label{def:sf}
  For $f \co R \hookrightarrow S$ a topological
  embedding of Riemann surfaces, the \emph{stretch factor} of~$f$
  is the maximal ratio of extremal lengths between the two surfaces:
  \begin{equation*}
    \SF[f] \coloneqq \sup_{C \in \Curves^+(R)}
      \frac{\EL_S[f(C)]}{\EL_R[C]},
  \end{equation*}
  where the supremum runs over all simple multi-curves~$C$ with
  $\EL_R[C] \ne 0$.
\end{definition}

We will show that $\SF[f]$ is achieved by a ratio of
extremal lengths of two measured foliations, not multi-curves. But $f$
does not induce a natural continuous map between measured foliations
(Example~\ref{examp:erase-hole}), so Definition~\ref{def:sf} is stated
in terms of multi-curves.

\begin{theorem}\label{thm:emb}
  Let $R$ and $S$ be Riemann surfaces and $f\co
  R \hookrightarrow S$ be a
  topological embedding so that no component of $f(R)$ is contained in
  a disk or a once-punctured disk. Then $f$
  is homotopic to a conformal embedding if and only if $\SF[f] \le 1$.
\end{theorem}

The key part of Theorem~\ref{thm:emb} is due to Ioffe
\cite{Ioffe75:QCImbedding}.
In fact, his results show that if $\SF[f] \ge 1$, it is related to the
quasi-conformal constant.

\begin{proposition}\label{prop:sf-qc}
  Let $f\co R \hookrightarrow S$ be a topological
  embedding of Riemann surfaces. If $\SF[f] \ge 1$, then
  $\SF[f]$ is equal to the smallest quasi-conformal constant of any
  quasi-conformal embedding homotopic to~$f$.
\end{proposition}

We can also characterize conformal embeddings with some extra ``room''.

\begin{definition}
  Let $f \co R \hookrightarrow
  S$ be a conformal embedding between Riemann surfaces. We say that $f$ is a
  \emph{strict} embedding if its image omits a
  non-empty open subset of each component of~$S$.
  An
  \emph{annular extension} of a Riemann surface $S$
  is a surface $\widehat S$ obtained by attaching a non-empty conformal
  annulus to each boundary component, with the boundary of $S$
  smoothly embedded in $\widehat S$. An \emph{annular
  conformal embedding} is one that extends to a conformal embedding
  $\widehat{R} \hookrightarrow S$ for some annular
  extension $\widehat{R}$ of~$R$.
\end{definition}

\begin{remark}
  A similar relation for subsets of $\CCa$ is sometimes written $f(R)
  \Subset S$ \cite[inter alia]{CPT16:Renorm}.
\end{remark}

\begin{theorem}\label{thm:strict-emb}
  Let $R$ and $S$ be Riemann surfaces, with $S$ connected, and let $f\co
  R \hookrightarrow S$ be a
  topological embedding so that no component of $f(R)$ is contained in
  a disk or a once-punctured disk. Then the following conditions are equivalent:
  \begin{enumerate}
  \item\label{item:strict-strict} $f$ is homotopic to a strict
    conformal embedding;
  \item\label{item:strict-annular} $f$ is homotopic to an annular
    conformal embedding;
  \item\label{item:strict-ball} there is a neighborhood $N$ of $S$ in
    Teichmüller space so that, for all $S' \in N$, $f$ is homotopic to a
    conformal embedding of $R$ in~$S'$; and
  \item\label{item:strict-sf} $\SF[f] < 1$.
  \end{enumerate}
\end{theorem}

\begin{remark}
  When $\SF[f] = 1$, the embedding guaranteed by Theorem~\ref{thm:emb}
  is instead a Teichmüller embedding in the sense of Definition~\ref{def:slit},
  with $K=1$ \cite{FB18:Couch}.
\end{remark}

In condition~(\ref{item:strict-ball}), $\SF[f]$ is related to the
size of the ball in Teichmüller space.

\begin{definition}\label{def:teich-embed}
  Let $f \co R \subset S$ be a topological embedding of Riemann
  surfaces. Let $\mathcal{T}_R(S)$ be the subset of the Teichmüller
  space $\mathcal{T}(S)$ for which there is a conformal embedding
  of~$R$ in the homotopy class~$[f]$. (This is empty if $\SF[f] > 1$.)
\end{definition}

\begin{proposition}\label{prop:sf-dist}
  Let $f\co R \hookrightarrow S$ be a topological
  embedding of Riemann surfaces, and suppose $\SF[f] \le 1$.
  Then
  \[
    d(S, \partial\Teich_R(S)) =
    -\frac{1}{2}\log \SF[f].
  \]
\end{proposition}

We can also control the behavior of the
stretch factor under
taking covers. Proposition~\ref{prop:sf-qc} guarantees that when
$\SF[f] \ge 1$, the stretch factor is unchanged under taking
finite covers (see Proposition~\ref{prop:SF-large-cover}).
We can control what happens when $\SF[f] < 1$, as
well.

\begin{definition}\label{def:cover-map}
  For $f \co R \hookrightarrow S$ a topological
  embedding of Riemann surfaces and $p \co \wt S \to S$ a covering
  map, the corresponding \emph{cover} of $f$ is
  the pull-back map~$\wt f$ in the diagram
  \[
  \begin{tikzpicture}
    \matrix[row sep=0.7cm,column sep=0.8cm] {
      \node (S1t) {$\wt R$}; &
        \node (S2t) {$\wt S$}; \\
      \node (S1) {$R$}; &
        \node (S2) {$S$.}; \\
    };
    \draw[->] (S1) to node[auto=left,cdlabel] {f} (S2);
    \draw[->,dashed] (S1t) to node[auto=left,cdlabel] {\wt f} (S2t);
    \draw[->] (S2t) to node[right,cdlabel] {p} (S2);
    \draw[->,dashed] (S1t) to node[right,cdlabel] {q} (S1);
  \end{tikzpicture}
  \]
  Explicitly, we have
  \begin{align*}
    \wt R &\coloneqq \bigl\{\,(r, \wt{s}) \in R \times \wt{S}
      \mid f(r) = p(\wt{s})\,\bigr\}\\
    \wt{f}(r,\wt{s}) &\coloneqq \wt{s}\\
    q(r,\wt{s}) &\coloneqq r.
  \end{align*}
  Then $\wt f$ is a topological embedding and $q$ is a
  covering map.
  We may also say that $\wt f$ is a cover of $f$, without
  specifying~$p$.
\end{definition}

\begin{definition}\label{def:lifted-stretch}
  For $f \co R \hookrightarrow S$ a topological
  embedding of Riemann surfaces, the \emph{lifted stretch factor}
  $\wt{\SF}[f]$ is
  \[
  \wt{\SF}[f] \coloneqq \sup_{\substack{\text{$\wt f$ finite}\\\text{cover of $f$}}} \SF[\wt f].
  \]
\end{definition}

\begin{theorem}\label{thm:sf-cover}
  Let $f \co R \hookrightarrow S$ be a topological
  embedding of Riemann surfaces. If $\SF[f] \ge 1$, then
  $\wt\SF[f] = \SF[f]$. If $\SF[f] < 1$, then
  \[
  \SF[f] \le \wt\SF[f] < 1.
  \]
\end{theorem}

The hard part of Theorem~\ref{thm:sf-cover} is showing that
$\wt\SF[f]$ is strictly less than $1$ when $\SF[f] < 1$.

By Proposition~\ref{prop:sf-dist}, $\wt \SF[f] < 1$ is equivalent to saying
that $\Teich_{\wt R}(\wt S)$ contains a ball of uniform size
around~$\wt S$ for every finite cover of~$f$.

Theorem~\ref{thm:sf-cover} will be used in later work
\cite{Thurston16:RubberBands}
to give a positive characterization of
post-critically finite rational maps among topological branched
self-covers of the sphere. This provides a counterpoint to W. Thurston's
characterization \cite{DH93:ThurstonChar}, which characterizes
rational maps in terms of an obstruction.

\subsection{History}
\label{sec:history}
The maximum of the ratio of extremal lengths has appeared before,
usually in the context of closed surfaces, where it gives
Teichmüller distance, as first proved by Kerckhoff
(see Theorem~\ref{thm:teich-dist} below). For surfaces with
boundary
the behavior is quite different, as the stretch factor can be less
than one.

In the special case when the
target~$S$ is a closed torus, there is very precise information about
when $R$ conformally embeds inside of~$S$
\cite{Shiba87:ModuliCont,Shiba93:SpansLowGenus}. Shiba proves that in
this case
$\Teich_R(S)$ is a disk with
respect to the Teichmüller metric.

There has been earlier work on portions of
Theorem~\ref{thm:strict-emb}. In
particular, Earle and
Marden \cite{EM78:ConfEmbeddings} showed that, with extra
topological restrictions on the embedding $R \hookrightarrow S$,
if $f$ is homotopic to a strict conformal embedding then it is
homotopic to an annular conformal embedding.

It is tempting to look for an analogue of Theorem~\ref{thm:emb} using
hyperbolic length instead of extremal length, given that, by the
Schwarz lemma, hyperbolic length is decreased under conformal
inclusion. However, the results are false for hyperbolic length in
almost all cases
\cite{Masumoto00:HypEmbedding,FB14:ConverseSchwarz}.

These results were first announced in a research report by the last
author \cite{Thurston16:RubberBands}.

\subsection{Organization}
\label{sec:organization}
Section~\ref{sec:setting} reviews background material and specify our
definitions for topological surfaces. Section~\ref{sec:conformal} does
the same for Riemann surfaces and extremal length, as well as giving
elementary properties of the stretch factor. Section~\ref{sec:slit}
proves Theorem~\ref{thm:emb}, largely based on a theorem of
Ioffe. Section~\ref{sec:strict} extends this to prove
Theorem~\ref{thm:strict-emb}. Section~\ref{sec:covers} gives the
further extension to prove Theorem~\ref{thm:sf-cover}. In
Section~\ref{sec:covers}, we also prove Theorem~\ref{thm:area-surface}, an
estimate on areas of subsurfaces with respect to quadratic
differentials; this may be of
independent interest. Section~\ref{sec:challenges} gives some
directions for future research, and in the process gives another way
to get an upper bound on $\wt\SF[f]$.

\subsection{Acknowledgments}
We thank
Matt Bainbridge,
Maxime Fortier Bourque, and
Frederick Gardiner
for many helpful conversations.
Aaron Cohen,
Russell Lodge,
Insung Park, and
Maxime Scott
gave useful comments on earlier drafts.
JK was supported by NSF grant DMS-1352721.
KMP was supported by Simons Foundation Collaboration Grant \#4429407.
DPT was supported by NSF grants DMS-1358638 and DMS-1507244.


\section{Topological Setting}
\label{sec:setting}
\begin{definition}\label{def:surface}
  By a (smooth) \emph{surface}~$S$ we mean a smooth, oriented, compact
  2-manifold
  with boundary, together with a distinguished finite set~$P$ of points
  in~$S$, the \emph{punctures}. The boundary $\partial S$ of~$S$ is
  a finite union of circles. By a slight abuse of terminology, by the
  interior~$S^\circ$ of~$S$ we mean $S \setminus (P \cup \bdy S)$.
  If we want to emphasize that we are talking about the
  compact version of~$S$, we will write $\closure{S}$.

  A surface is \emph{small} if it is the sphere with $0$, $1$,
  or~$2$ punctures or the unit disk with $0$ or~$1$ punctures. These
  are the surfaces that have no non-trivial curves by the definition
  below.

\  By a \emph{topological map} $f \co R \to S$ between surfaces we mean an
  orientation\hyp preserving continuous map from
  $R^\circ$ to~$S^\circ$ that extends to a continuous map from $\closure{R}$ to~$\closure{S}$. In
  particular, the image of a puncture is a puncture or a regular
  point, and embeddings are only required to be one-to-one on $R^\circ$.
  Homotopies are taken within the same space of maps.
\end{definition}

\begin{definition}\label{def:curves}
  A \emph{multi-curve}~$C$ on a surface~$S$ is a smooth 1-manifold with
  boundary $X(C)$ together with an immersion from the interior of $X(C)$
  into~$S^\circ$ that maps $\partial X(C)$ to
  $\partial S$. We do not assume that $X(C)$ is connected; if it
  is, $C$ is said to be \emph{connected} or a
  \emph{curve}. We will mostly be concerned with \emph{simple}
  multi-curves, those for which the immersion is an embedding. An \emph{arc}
  is a curve for which $X(C)$ is an interval, and a \emph{loop} is a
  curve for which $X(C)$ is a circle. A multi-curve is \emph{closed}
  if it has no arc components.

  A (multi-)curve is \emph{trivial} if
  it is contained in a disk or once-punctured disk of~$S$.

  \emph{Equivalence} of multi-curves is the equivalence relation generated
  by
  \begin{enumerate}
  \item homotopy within the space of all maps taking $\partial X(C)$ to
    $\partial S$ (not necessarily immersions),
  \item reparametrization of the 1-manifold $X(C)$ (including
    orientation reversal), and
  \item dropping trivial components.
  \end{enumerate}
  The equivalence class of~$C$ is denoted~$[C]$. The space of simple
  multi-curves on~$S$ up to homotopy is denoted $\Curves^{\pm}(S)$. If
  $\bdy S \ne \emptyset$, then we distinguish two subsets of
  $\Curves^{\pm}(S)$:
  \begin{itemize}
  \item $\Curves^+(S) \subset \Curves^{\pm}(S)$ is the subset of closed
    curves and
  \item $\Curves^-(S) \subset \Curves^{\pm}(S)$ is the subset with no loops
    parallel to a boundary component.
  \end{itemize}

  A \emph{weighted multi-curve} $C = \sum a_iC_i$ is a multi-curve in which each
  connected component is given a positive real coefficient~$a_i$. When
  considering equivalence of weighted multi-curves, we add the further
  relation that two parallel components may be merged and their
  weights added. We write $\Curves_\RR(S)$ or
  $\Curves_\QQ(S)$ for the space of weighted multi-curves with real or
  rational weights, respectively.
\end{definition}

\begin{definition}\label{def:foliation}
  A (positive) \emph{measured foliation} $F$ on a surface $S$ is
  a singular 1-dimensional foliation on $\overline{S}$, tangent
  to~$\bdy S$, with a non-zero transverse
  measure. $F$ is allowed to have $k$-prong singularities, as
  described, for instance, in
  \cite{FLP79:TravauxThurston}, and summarized below.
  \begin{itemize}
  \item At points of $S^\circ$, we allow $k$-prong singularities for $k
    \ge 3$. (If there are only $2$ prongs, it is not a singularity.)
    This is also called a zero of order $k-2$.
  \item At punctures, we allow $k$-prong singularities for $k \ge
    1$. This is also called a zero of order $k-1$.
  \item At points of~$\bdy S$, we allow $k$-prong singularities for
    $k \ge 3$. This is also called a zero of order $k-2$. If we double
    the surface, it becomes a $(2k-2)$-prong singularity.
  \end{itemize}
  We also admit the empty (zero) measured foliation as a degenerate case.
  A \emph{singular leaf} of a measured foliation is a leaf that ends
  at a singularity. A
  \emph{saddle connection} is a
  singular leaf that ends at singularities in both
  directions. If a saddle connection
  connects two distinct singularities, and at least one of the
  singularities is in
  the interior, it is possible to \emph{collapse} it to form a new
  measured foliation.
  \emph{Whitehead equivalence} of measured foliations is
  the equivalence relation generated by homotopy and collapsing saddle
  connections. We denote the Whitehead equivalence class of a
  measured foliation by~$[F]$, and the set of Whitehead equivalence
  classes of measured foliations by $\MF^+(S)$.
\end{definition}

From a multi-curve $C \in \Curves^-(S)$ and a measured foliation $F$
on~$S$, we can form the intersection number
\begin{equation}\label{eq:mf-intersect}
  i([C],[F]) \coloneqq \inf_{C_1 \in [C]} \int_t \abs{F(C_1'(t))}\,dt.
\end{equation}

\begin{proposition}\label{prop:mf-measure}
  The map
  \begin{align*}
    \MF^+(S) &\rightarrow \RR^{\Curves^-(S)}\\
    [F] &\mapsto \bigl(i([C],[F])\bigr)_{[C] \in \Curves^-(S)}
  \end{align*}
  is an injection, with image a finite-dimensional manifold determined
  by its projection onto finitely many factors.
\end{proposition}
\begin{proof}[Proof sketch]
  This is standard. If $S$ has non-empty boundary, take a
  maximal set $(C_i)_{i=1}^n$ of non-intersecting arcs in
  $\Curves^-(S)$. Then $(i([C_i],[F]))_{i=1}^n$ determines~$F$ up
  to Whitehead equivalence. If $S$ has no boundary, the construction
  is more involved, and we omit it.
\end{proof}

Proposition~\ref{prop:mf-measure} can be used to define a topology on
$\MF^+(S)$, which we will use.

\begin{proposition}
  The projection map from all measured foliations (not up to
  equivalence, with its natural function topology) to $\MF^+(S)$ is
  continuous.
\end{proposition}
\begin{proof}[Proof sketch]
  For any non-zero measured foliation~$F_0$ and $[C] \in \Curves^-(S)$, there
  is a quasi-transverse representative $C_0\in [C]$, which
  automatically satisfies $i(C_0,F_0) = i([C],[F_0])$. If $F_1$ is any
  measured foliation close to~$F_0$, then an analysis of the behavior
  near singularities shows that there is a representative $C_1\in [C]$
  so that $C_1$ is close to $C_0$ and $C_1$ is quasi-transverse with
  respect to $F_1$. Then $i([C],[F_1]) = i(C_1, F_1)$ and $i(C_1, F_1$
  is close to $i(C_0, F_0)$.
\end{proof}

We can also use Proposition~\ref{prop:mf-measure} to define a map from
$\Curves^+(S)$ to $\MF^+(S)$, sending $[C] \in \Curves^+(S)$ to the unique
measured foliation $[F_C]\in\MF^+(S)$ so that $i([C'],[C]) =
i([C'],[F_C])$ for all $C' \in \Curves^-(S)$. This map is an embedding
on equivalence classes of weighted simple multi-curves.

\begin{definition}
  A \emph{train track}~$T$ on a surface~$S$ is a graph $G$ embedded
  in~$S$, so that at each vertex of~$G$ (called a \emph{switch}) the
  incident edges are partitioned into two non-empty subsets that are
  non-crossing in the cyclic order on the incident vertices. In
  drawings, the elements of each subset are drawn tangent to each
  other.

  The \emph{complementary regions} of a train track are naturally
  surfaces with cusps on the boundary.
  A \emph{taut} train track is a train track with no complementary
  components that are disks with no cusps or one cusp, or once-punctured
  disks with no cusps.
\end{definition}

\begin{remark}
  Many authors (e.g., Penner and Harer \cite{PH92:CombTrainTracks} and
  Mosher \cite{Mosher03:Arational})
  include our notion of tautness in the definition of a train
  track, often in a stronger form
  forbidding bigons (disks with two cusps) and once-punctured monogons as well.
\end{remark}

\begin{definition}
  The space of positive \emph{transverse measures} or \emph{weights} on a train
  track~$T$ on a surface~$S$ is the space
  $\Meas(T)$ of assignments of positive numbers (``widths'') to edges of the
  train track so that, at each vertex, the sum of weights on the two
  sides of the vertex are equal. If $\Meas(T)$ is non-empty, then
  $T$ is said to be \emph{recurrent}. We have subspaces
  $\Meas_{\QQ}(T)$ and $\Meas_{\ZZ}(T)$ for transverse measures on~$T$
  with rational or integral values, respectively. For any train
  track, there
  is a natural map $\Meas_{\ZZ}(T) \to \Curves^+(S)$, where we replace
  an edge of~$T$ of weight~$k$ by $k$ parallel strands, joining the
  strands in the natural way at the switches.
\end{definition}

\begin{lemma}\label{lem:tt-continuous}
  Let $T$ be a recurrent taut train track on~$S$. Then there is a
  natural continuous map $\Meas(T) \to
  \MF^+(S)$ extending $\Meas_{\ZZ}(T) \to \Curves^+(S)$.
\end{lemma}

We will denote the map $\Meas(T) \to \MF^+(S)$ by $w \mapsto T(w)$. If
$F = T(w)$ for some~$w$, we say that $T$ \emph{carries}~$F$.

(For convenience in the proof we are assuming the weights on~$T$ are
strictly positive, but in fact the lemma extends to non-negative
weights.)

\begin{proof}
  Pick a small regular neighborhood $N(T)$ of~$T$, arranged so that
  $S \setminus N(T)$ has a cusp near each corner where $T$ has a
  cusp, as illustrated in Figure~\ref{fig:sew-train-track}. A weight
  $w \in \Meas(T)$ gives a canonical measured
  foliation $F_N(w)$ on $N(T)$, where an arc cutting across $N(T)$ transverse
  to a edge~$e$ has measure $w(e)$.

  Next pick a graph $\Gamma \subset \overline{S} \setminus N(T)$ so that
  \begin{itemize}
  \item $\Gamma$ contains $\partial S$,
  \item $\Gamma$ has a 1-valent vertex at each cusp of $S \setminus
    N(T)$ and at each puncture,
  \item all other vertices of~$\Gamma$ have valence $2$ or more, and
  \item $\Gamma$ is a spine for $S \setminus N(T)$, i.e., $S \setminus
    N(T)$ deformation retracts onto~$\Gamma$.
  \end{itemize}
  (The condition that $T$ be taut guarantees that we can find such a
  $\Gamma$). Since $\Gamma$ is a spine, there is a deformation
  retraction $\overline{S} \setminus N(T) \to \Gamma$. We can use this to
  construct a homeomorphism $\phi \co N(T) \to \overline{S} \setminus
  \Gamma$ that is the identity on $T \subset N(T)$ and extends to a
  continuous map $\partial N(T) \to \Gamma$ without backtracking. Then
  $[\phi(F_N(w))]$ is the desired measured foliation $T(w)$.

  \begin{figure}
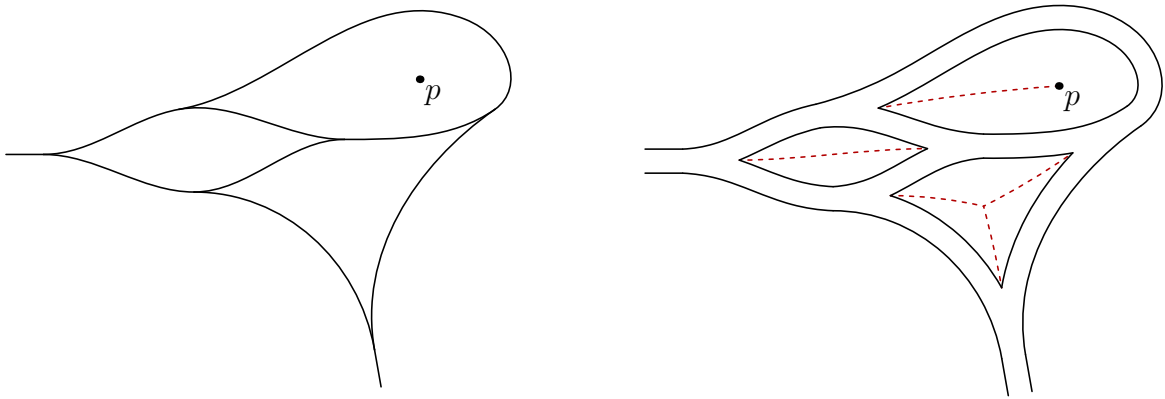

    \[\mfigb{tt-0}\qquad\qquad\mfigb{tt-2}\]
    \caption{Left: A portion of a taut train track~$T$.  The
      point~$p$ is a puncture. Right: A
      neighborhood $N(T)$ of~$T$, together with a spine~$\Gamma$
      for $\overline{S} \setminus T$, shown dashed in red.}
    \label{fig:sew-train-track}
  \end{figure}

  As a measured foliation (not up to Whitehead equivalence),
  $\phi(F_N(w))$ depends continuously on~$w$ by construction. The
  quotient map to the Whitehead equivalence class is continuous.
\end{proof}

In Lemma~\ref{lem:tt-continuous}, if a complementary
region of~$T$ is a bigon or once-punctured monogon, the corresponding spine
is necessarily an interval. Lemma~\ref{lem:tt-continuous} is false
without the assumption that $T$ is taut; see
Example~\ref{examp:erase-hole}.

\begin{lemma}\label{lem:MF-carried}
  Every measured foliation~$F$ is carried by a taut train
  track~$T$. Furthermore, $T$ can be chosen so that if $F$ has
  $k$~zeros on a
  boundary component, the
  corresponding complementary component of~$T$ has at least
  $k$~cusps.
\end{lemma}

Notice that the number of zeros on a boundary component is not
invariant under Whitehead equivalence.

\begin{proof}
  The techniques here are standard; see, e.g., \cite[Proposition
  3.6.1]{Mosher03:Arational}, or \cite[Corollary
  1.7.6]{PH92:CombTrainTracks} for a different approach.
  Since the definitions we use are slightly
  different, we sketch the argument.

  Pick a set of intervals $I_j$ on~$S$ that are transverse
  to~$F$ and cut every leaf of~$F$. These intervals will become
  the switches of the train track. Let $I = \bigcup_j I_j$.

  Divide the leaves of~$F$ into \emph{segments} between singularities
  of~$F$ and
  intersections with~$I$. A \emph{regular} segment is one that
  intersects $I$ in interior points on both ends. There are only a
  finite number of non-regular segments (since the
  number of singularities of~$F$ and ends of~$I$ is finite), while for
  any regular
  segment, nearby segments are isotopic relative to~$I$. There are
  thus a finite number of classes of parallel regular segments.

  Now construct a train track~$T$ by taking the union of~$I$ and one
  element of each class of parallel regular segments, and replacing
  each interval $I_j$ with a single vertex~$v_j$, joined to the
  same regular segments by connecting arcs. At each
  switch, the incident edges are divided according to the sides of the
  corresponding~$I_j$.

  Let $\Gamma$ be the union of the non-regular segments. The components of
  the complement of~$T$ correspond to the connected components of~$\Gamma$,
  which is a graph with vertices of valence~$1$ at cusps of~$T$ and
  possibly at punctures of~$S$, and all other vertices of valence $\ge
  2$. (That is, $\Gamma$ is a spine as in the proof of
  Lemma~\ref{lem:tt-continuous}.)  It follows that $T$ is taut.
  $T$~carries $F$ with weights
  equal to the width of the families of parallel segments. If $F$ has
  $k$~zeros on a boundary component, then $T$ has at least $k$~cusps
  by construction.
\end{proof}

\begin{proposition}\label{prop:curves-dense}
  $\Curves^+_\QQ(S)$ is dense in $\MF^+(S)$.
\end{proposition}

\begin{proof}
  For a given measured foliation~$F$, we
  will produce a
  sequence of weighted multi-curves approximating $[F] \in \MF^+(S)$. By
  Lemma~\ref{lem:MF-carried}, $F = T(w)$ for a taut train track~$T$
  and weight $w \in \Meas(T)$. Pick a sequence of
  rational weights $w_n \in \Meas_{\QQ}(T)$ approximating~$w$, and
  clear denominators to write $w_n = \lambda_n w_n'$ where $w_n' \in
  \Meas_{\ZZ}(T)$. Then $w_n'(T)/\lambda_n$ is a weighted multi-curve
  approximating~$F$.
\end{proof}

\begin{remark}\label{rem:connected}
  On a connected surface~$S$ with no boundary,
  Proposition~\ref{prop:curves-dense} can be strengthened to say that
  simple curves are projectively dense in measured
  foliations, as well
  \cite{Kerckhoff80:AsympTeich,FLP79:TravauxThurston}. This
  strengthening is false for surfaces with boundary. For
  instance, a pair of pants has only three distinct non-trivial
  simple curves, but a 3-dimensional space of measured foliations.
\end{remark}


\section{Conformal Setting\label{sec:conformal}}

\subsection{Riemann surfaces\label{sec:riemann}}

\begin{definition}
  A \emph{Riemann surface} (with boundary) is a smooth
  surface~$S$, as in
  Definition~\ref{def:surface}, together with
  a complex structure on~$\oS$, i.e., a fiberwise linear map $J \co
  T\oS \to T\oS$ with $J^2 = -\textrm{id}$.
\end{definition}

\begin{convention}
  For us, a Riemann surface need not be connected. We only consider
  surfaces of finite topological type.
\end{convention}

Since the complex structure is on $\oS$, not just on~$S$, the complex
structure on $S^\circ$ necessarily has a removable singularity near
every puncture.

\begin{definition}
  A (holomorphic) \emph{quadratic differential} $q$ on a Riemann
  surface~$S$ is a holomorphic
  section of the square of the holomorphic cotangent bundle
  of~$S^\circ$. If $z$ is a local coordinate on
  $S^\circ$, we can
  write $q = \phi(z)\,(dz)^2$ where $\phi(z)$ is holomorphic.

  Naturally associated with a quadratic differential we have several objects:
  \begin{itemize}
  \item Local coordinates given by integrating a branch of $\sqrt{q}$
    away from the zeros of~$q$. The transition maps are translations
    or half-turns followed by translations, giving $S$ the structure
    of a half-translation surface.
  \item A horizontal measured foliation $\Fh(q) = \abs{\Im \sqrt{q}}$.
    The tangent vectors
    to the foliation are those vectors $v\in TS$ with $q(v) \ge 0$, and the
    transverse length of a multi-curve~$C$ is
    \[
    \Fh(q)(C) = \int_{t} \bigl\lvert\Im \sqrt{q(C'(t))}\bigr\rvert\,dt,
    \]
    i.e., the total variation of the $y$ coordinate in the
    half-translation coordinates.
  \item Similarly, a vertical measured foliation $\Fv(q) = \abs{\Re \sqrt{q}}$.
  \item A locally Euclidean metric $\abs{q}$ on~$S^\circ$, possibly
    with cone singularities of cone angle $k\pi$ with $k \ge 3$.
    The
    length of a multi-curve~$C$ with respect to $\abs{q}$ is
    \[
    \ell(C) = \int_t \sqrt{\abs{q(C'(t))}}\,dt.
    \]
  \item An area measure $A_q$ on~$S$, the volume measure of $\abs{q}$.
  \end{itemize}

  The vector space of finite-area quadratic differentials on~$S$ that
  extend analytically to $\bdy S$ (but not necessarily to the punctures)
  is denoted $\Quad(S)$. The finite area constraint implies that at a
  puncture of~$S$, every $q\in\Quad(S)$ has at most a simple pole.
  That is, if $z$ is a local coordinate on $\oS$ with a puncture at
  $z=0$, we can locally write $q = \phi(z)/z\,(dz)^2$ where $\phi(z)$
  is holomorphic. 

If $S$ has non-empty boundary, then $\Quad(S)$ is
  infinite\hyp dimensional. There is a finite\hyp dimensional subspace $\Quad^\RR(S)$,
  consisting of those quadratic differentials that are real on vectors
  tangent to $\partial S$. There is a convex cone $\Quad^+(S) \subset
  \Quad^\RR(S)$ consisting of those quadratic differentials that are
  non-negative on $\partial S$. For non-zero $q$ in $\Quad^+(S)$, we have
  $[\Fh(q)] \in \MF^+(S)$.
\end{definition}

\subsection{Extremal length}
\label{sec:ext-length}
\begin{definition}\label{def:el-curves}
  For $C$ a multi-curve on a Riemann surface~$S$, pick a Riemannian
  metric~$g$ in the distinguished conformal class. Then the \emph{length}
  $\ell_g[C]$ is the minimum Riemannian
  length with respect to~$g$ of any rectifiable representative of~$[C]$. The
  \emph{extremal length} of $C$ is
  \begin{equation}\label{eq:el-sup}
  \EL_S[C] \coloneqq \sup_\rho \frac{\ell_{\rho g}[C]^2}{A_{\rho g}(S)},
  \end{equation}
  where the supremum runs over all Borel-measurable
  conformal
  scale factors $\rho \co S \to \RR_{\ge 0}$ of finite, positive
  area. (The scaled quantity $\rho g$ may give a pseudo-metric rather
  than a metric, as, e.g., $\rho$ can be~$0$ on an open subset
  of~$S$. But we can still define length
  and area in a natural way.)  The definition makes it clear that
  extremal length does
  not depend on the metric $g$ within its conformal equivalence class,
  so we may talk about extremal length on~$S$ without reference to~$g$.

  When the Riemann surface is clear from context, we suppress it
  from the notation.

  More generally, if $C = \sum_i a_i C_i$ is a weighted multi-curve, then its
  length is $\ell_g[C] = \sum_i a_i \ell_g[C_i]$, i.e., the
  corresponding weighted linear
  combination of lengths of curves, and its extremal length is still defined by
  Equation~\eqref{eq:el-sup}.
\end{definition}

We need multi-curves in Definition~\ref{def:el-curves}, as the main
theorems of this paper
are false if restricted to curves rather than multi-curves; see
Remark~\ref{rem:connected}.

We will be interested in simple multi-curves~$C$.
We must check that extremal length is well-defined on equivalence
classes of simple multi-curves. Invariance under homotopy and
reparametrization is obvious. Trivial components of a multi-curve~$C$
have no effect on $\ell_{\rho g}[C]$ since $\rho g$ has finite
area, so also have no effect on extremal length. Finally, let $C_0$ be
a simple
multi-curve with parallel components, and let $C_1$ be the weighted
multi-curve with integer weights obtained by merging parallel
components and taking the weight to be the number of merged
components. Then it is easy to see from the definitions that $\EL[C_0]
= \EL[C_1]$.

Furthermore, $\EL$ scales
quadratically: $\EL[kC] = k^2\EL[C]$. 

\begin{lemma}\label{lem:EL-small}
  For any non-trivial multi-curve~$C$ on a Riemann surface~$S$,
  $\EL[C] > 0$.  In particular, if $S$ is not small, there
  is a curve with non-zero extremal length.
\end{lemma}

\begin{proof}
  Take any finite-area Riemannian metric~$g$ on~$S$ in the given
  conformal
  class. Then, since $C$ has at least one non-trivial component,
  $\ell_g[C] > 0$, so $\EL[C] > 0$.
\end{proof}

We next give some
other interpretations of extremal length for simple multi-curves.
First, recall that for a conformal annulus
\[
A = \bigl([0,\ell] \times [0,w]\bigr)/\bigl((0,x) \sim (\ell,x)\bigr),
\]
its \emph{modulus} $\Mod(A)$ is $w/\ell$. Define the \emph{extremal
  length} of $A$ to be $\EL(A) \coloneqq 1/\Mod(A) = \ell/w$.
Then we can see $\EL[C]$ for a simple multi-curve~$C$ as finding the fattest
set of conformal annuli around~$C$, in the sense that we minimize
total extremal length, as follows.

\begin{proposition}\label{prop:el-total-el}
Let $C = \bigcup_{i=1}^k C_i$ be a simple closed multi-curve on a Riemann
surface~$S$ with
components~$C_i$.
For $i=1,\dots,k$, let
$A_i$ be a (topological) annulus, and let $A = \bigcup_{i=1}^k A_i$.
Then
\begin{equation}
\EL[C] = \inf\limits_{\omega, f}\,\,\sum\limits_{i=1}^k\,\, \EL_\omega(A_i),
\end{equation}
where the infimum runs over all conformal structures~$\omega$ on~$A$
(which amounts to a choice of modulus for each $A_i$) and over all
conformal embeddings $f\co A \hookrightarrow S$ so that the image of the core
curve of~$A_i$ is isotopic to~$C_i$.

More generally, if $C = \sum_{i=1}^k a_i C_i$ is a weighted simple
multi-curve on~$S$, then, with notation as above,
\begin{equation}\label{eq:el-total-el}
\EL[C] = \inf_{\omega,f}\,\,\sum\limits_{i=1}^k\,\,
  a_i^2 \EL_\omega(A_i),
\end{equation}
where the supremum runs over the same set.
\end{proposition}

We delay the proof of Proposition~\ref{prop:el-total-el} a little.

We can also give a characterization of~$\EL$ in terms of Jenkins-Strebel
differentials.

\begin{definition}
  A \emph{Jenkins-Strebel} quadratic differential~$q$ on~$S$ is one
  where almost every leaf of $\Fh(q)$ is closed. In this case, the
  quadratic differential gives a canonical decomposition of~$S$
  into annuli foliated by the closed leaves.
\end{definition}

\begin{citethm}\label{thm:quad-diff-height}
  Let $C=\bigcup_i a_i C_i$ be a weighted simple closed multi-curve on a Riemann
  surface~$S$ so that no $C_i$ is trivial. Then there is a
  unique Jenkins-Strebel
  differential~$q_C\in\Quad^+(S)$ so that $\Fh(q_C)$ can be
  decomposed as a disjoint union of annuli~$A_i$ with each $A_i$ being
  a union of leaves of transverse measure~$a_i$ and core curve
  isotopic to~$C_i$. With
  respect to
  $\abs{q_C}$, each $A_i$ is isometric to a right Euclidean cylinder.
\end{citethm}

For a proof, see, e.g., Strebel \cite[Theorem
21.1]{Strebel84:QuadDiff}, who attributes the theorem to
Hubbard-Masur \cite{HM79:QuadDiffFol} and Renelt \cite{Renelt76:QD}.
This theorem is one of three different
standard theorems on the existence of Jenkins-Strebel differentials.

\begin{proposition}\label{prop:el-area}
  For $C$ a weighted simple closed multi-curve on~$S$ with no trivial
  components, let $q = q_C$ be the associated quadratic differential from
  Theorem~\ref{thm:quad-diff-height}. Then
  \begin{equation}
    \EL[C] = A_{\abs{q}}(S).
  \end{equation}
\end{proposition}

Proposition~\ref{prop:el-area} should be standard, but we have been
unable to locate it in the literature. We provide a short proof,
an easy application of Beurling's criterion.

\begin{proof}
  We use $\abs{q}$ as the base metric in Equation~\eqref{eq:el-sup}
  (abusing notation slightly since $\abs{q}$ is not smooth).  Let
  $\ell_i = \ell_{\abs{q}}[C_i]$.  Since $\abs{q}$ is a locally CAT(0)
  metric and local geodesics in locally CAT(0) spaces are globally
  length-minimizing, $\ell_i$ is the length in~$\abs{q}$ of the core
  curve of the annulus~$A_i$.
  (This also follows from Teichmüller's Lemma \cite[Theorem
  14.1]{Strebel84:QuadDiff}.)
  Then, since
  $A_{q}(S) = \sum_i a_i \ell_i$ by the construction of~$q$ and
  $\ell_{\abs{q}}[C] = \sum_i a_i \ell_i$ by definition of $\ell_{\abs{q}}[C]$,
  \[
  \frac{\ell_{\abs{q}}[C]^2}{A_{q}(S)}
  = A_{q}(S),
  \]
  so $\EL[C] \ge A_{q}(S)$.

  For the other direction, let $\rho$ be the scaling factor relative
  to~$\abs{q}$ for another metric in the conformal class. For each $i$ and
  $t \in [0,a_i]$, let $C_i(t)$ be the curve on~$A_i$ at distance~$t$
  from one of boundary, let $s_i(t) = \int_{C_i(t)}\rho(x)\,dx$,
  and let
  $S_i = \min_{t\in[0,a_i]} s_i(t)$. Then, using the Cauchy-Schwarz
  inequality, we have
  \begin{align*}
    \ell_{\rho\abs{q}}[C] &\le \sum\nolimits_i a_i S_i\\
    A_{\rho\abs{q}}(S) &= \iint\nolimits_S \rho^2\,dA_q
       \ge \frac{1}{A_{q}(S)} \left(\iint\nolimits_S \rho\,dA_q\right)^2
       \ge \frac{1}{A_{q}(S)} \left(\sum\nolimits_i a_i S_i\right)^2\\
    \frac{\ell_{\rho\abs{q}}[C]^2}{A_{\rho\abs{q}}(S)}
       &\le A_{q}(S).\qedhere
  \end{align*}
\end{proof}

\begin{proof}[Proof of Proposition~\ref{prop:el-total-el}]
  The functional $\sum_{i=1}^k a_i^2 \EL(A_i)$ on the space of
  disjoint embeddings of annuli $A_i$ homotopic to $C_i$
  is minimized when the $A_i$ are the annuli from the decomposition of
  $\Fh(q_C)$
  from Theorem~\ref{thm:quad-diff-height} \cite[Theorem
  20.5]{Strebel84:QuadDiff}. In this case the value of
  the functional is $A_{q_C}(S)$, which is equal to $\EL[C]$ by
  Proposition~\ref{prop:el-area}.
\end{proof}

More generally, we can work with arbitrary measured foliations, rather
than simple multi-curves.

\begin{citethm}[Heights theorem]\label{thm:heights}
  Let $[F] \in \MF^+(S)$ be a measured foliation on a Riemann
  surface~$S$. Then there is a unique quadratic differential $q_F \in
  \Quad^+(S)$ so that $[\Fh(q_F)] = [F]$. Furthermore, $q_F$ depends
  continuously on~$F$.
\end{citethm}

Proofs of Theorem~\ref{thm:heights} have been
given by many authors
\cite{HM79:QuadDiffFol,Kerckhoff80:AsympTeich,MS84:Heights,Wolf96:RealizingMF}.
Of these, Marden and Strebel \cite{MS84:Heights} treat surfaces with
boundary.
By analogy
with Proposition~\ref{prop:el-area}, we define
\begin{equation}
  \EL[F] \coloneqq A_{q_F}(S).
\end{equation}
$\EL[F]$ can also be given a definition closer to
Definition~\ref{def:el-curves}
\cite[Section~4.4]{McMullen12:RS-course}.

\subsection{Stretch factors}
\label{sec:stretch-factors}

We now turn to a few elementary facts about stretch factors, as
already defined in Definition~\ref{def:sf}.

\begin{proposition}
  If $f \co R \hookrightarrow S$ is an topological embedding of
  Riemann surfaces where $R$ is not a small surface, then $\SF[f]$ is
  defined and finite.
\end{proposition}

\begin{proof}
  Immediate from Lemma~\ref{lem:EL-small}.
\end{proof}

\begin{definition}
  By analogy with Definition~\ref{def:curves}, we say that a
  subsurface $S'$ of a surface~$S$ is \emph{trivial} if~$S'$ is
  contained in a disk or once-punctured disk inside~$S$.
\end{definition}

\begin{proposition}
  For $f \co R \hookrightarrow S$ a topological embedding of Riemann
  surfaces where $S$ is not small, $\SF[f] = 0$ if and only if $f(R)$
  is trivial as a
  subsurface of~$S$.
\end{proposition}

\begin{proof}
  If $f(R)$ is trivial in~$S$, it is immediate that $SF[f] =
  0$. Otherwise, there is some simple curve~$C$ on~$R$ so that $f(C)$ is
  nontrivial in~$S$. It follows that $C$ is nontrivial in~$R$, and $\SF[f] \ge
  \EL[f(C)]/\EL[C] > 0$.
\end{proof}

\begin{proposition}\label{prop:sf-compose}
  If $f\co S_1 \hookrightarrow S_2$ and $g \co S_2
  \hookrightarrow S_3$ are two topological embeddings of
  Riemann surfaces, then
  \[
  \SF[f \circ g] \le \SF[f] \cdot\SF[g].
  \]
\end{proposition}

\begin{proof}
  Immediate from the definition.
\end{proof}

\subsection{Teichmüller space}
\label{sec:teichmuller-space}

We can assemble the Riemann surface structures on an underlying
smooth surface~$S$ into the (reduced) \emph{Teichmüller space}
$\Teich(S)$, meaning the space of Riemann surfaces $T$ together with a
homeomorphisms $\phi_T \co S \to T$, considered up to isotopies,
taking the boundary to itself but not required to fix it pointwise.
The \emph{Teichmüller distance} between two points in $\Teich(S)$ is
defined by
\[
  d(T,T') \coloneqq \frac{1}{2} \log K,
\]
where $K$ is the minimal stretching of any quasi-conformal homeomorphism~$f$
from $T$ to~$T'$ so that $(\phi_{T'})^{-1} \circ f \circ \phi_T$ is
isotopic to the identity. (Note that this definition uses
homeomorphisms, rather than the embeddings used in most of the paper.)

It is a standard result that there is a map~$f$ realizing the minimal
stretching~$K$, and that its Beltrami differential has the form
\begin{equation}\label{eq:quad-beltrami}
  \mu_f = \frac{K-1}{K+1} \frac{\overline{q}}{\abs{q}}
\end{equation}
for some quadratic differential $q \in q^{\RR}(T)$. Concretely, we
stretch the Euclidean metric $\abs{q}$ along $\mathcal{F}_h(q)$ by a
factor of~$K$. (Since $q$ is only real and not positive on
$\partial T$, $\mathcal{F}_h(q)$ will not in general be in $\MF^+(T)$.)
This is usually stated and proved for closed surfaces;
the case with boundary follows by considering $T \cup \overline{T}$,
the double of $T$ along its boundary, solving the problem in that
context, and observing that the optimal map~$f$ (which is usually
unique) can be chosen to be equivariant with respect to the
anti-holomorphic involution that interchanges $T$ and~$\overline{T}$
so must be real on $\partial T$.

It follows from Equation~\eqref{eq:quad-beltrami} that
\[
  \frac{\EL_{T'}(f_* \mathcal{F}_h(q))}{\EL_T(\mathcal{F}_h(q)} = K,
\]
and that this is the maximal ratio of extremal lengths. We can
approximate $\mathcal{F}_h(q)$ by a weighted multi-curve, possibly
with some arc components. We can therefore write the distance in terms
of ratios of extremal lengths. If $f \co T \to T'$ is a homeomorphism,
define a version of the stretch factor by
\[
  \SF^{\pm}[f] \coloneqq \sup_{C \in \mathcal{C}^\pm_{\RR}}
  \frac{\EL_{T'}[f(C)]}{\EL_T[C]}.
\]
That is, we allow arc components of the weighted multi-curve; extremal
length extends in the natural way to these multi-curves. If $C$ has
arc components, $f(C)$ is only well-defined since $f$ is a
homeomorphism. We have $\SF[f] \le \SF^{\pm}[f]$, since the
supremum is over a larger set.
\begin{citethm}\label{thm:teich-dist}
  The Teichmüller distance between $T,T' \in \Teich(S)$ is
  \[
    d(T, T') = \frac{1}{2}\log \SF^{\pm}[\id_{T,T'}].
  \]
\end{citethm}
Theorem~\ref{thm:teich-dist} was stated and proved by Kerckhoff
\cite[Theorem~4]{Kerckhoff80:AsympTeich} for closed surfaces. 
He furthermore restricted to simple curves (not multi-curves); the
technique for the reduction to simple curves cannot be made
equivariant with respect to the map interchanging the two components
of the mirror of~$T$.


\section{Slit maps and Ioffe's theorem}
\label{sec:slit}

The following terminology is adapted from Ioffe
\cite{Ioffe75:QCImbedding} and Fortier Bourque \cite{FB18:Couch}.

\begin{definition}\label{def:slit}
  On a connected surface $S$ with a non-zero measured
  foliation~$F$, a \emph{slit} is a finite union of
  closed segments of leafs of~$F$. (The leaf segments can meet
  at singularities of~$F$, and so the slit may be a graph.)
  A \emph{slit complement} in~$F$
  is the complement of a slit, and a \emph{topological slit map} with
  respect to~$F$ is
  the inclusion of a slit complement into $S$. (This is the inclusion
  of the interiors $R^\circ \hookrightarrow S^\circ$, which extends on
  a non-injective map $\closure{R} \to \closure{S}$.)

  If $f \co R \hookrightarrow S$ is a slit map with
  respect to $F \in \MF^+(S)$, then there is a natural pull-back
  measured foliation $f^* F \in \MF^+(R)$.

  If $R$ and $S$ are Riemann surfaces, a \emph{Teichmüller
    embedding} of dilatation $K \ge 1$ is an embedding $f\co
  R \hookrightarrow S$ with quadratic differentials $q_R\in
  \Quad^+(R)$ and
  $q_S \in \Quad^+(S)$ so that $f$ is a topological slit map with respect
  to $\Fh(q_S)$ and, in the natural coordinates determined by $q_R$
  and~$q_S$, the map $f$ has the form $f(x+iy) = Kx+iy$. Note
  that a Teichmüller embedding is $K$-quasi-conformal, and that
  $f^* \Fh(q_S) = \Fh(q_R)$.
\end{definition}

\begin{citethm}[Ioffe \cite{Ioffe75:QCImbedding}]\label{thm:ioffe}
  Let $R$ and $S$ be Riemann surfaces, with $S$
  connected, and let $f\co R \hookrightarrow S$ be a
  topological embedding so that no component of $R$ has trivial
  image in~$S$. Suppose that $f$ is not homotopic to a
  conformal embedding. Then there is a quasi-conformal embedding with
  minimal dilatation in $[f]$. Furthermore, there are unique quadratic
  differentials $q_R \in \Quad^+(R)$ and $q_S \in \Quad^+(S)$ so that
  all quasi-conformal embeddings with minimal dilatation are Teichmüller
  embeddings with respect to the same quadratic differentials on
  $R$ and~$S$.
\end{citethm}

\begin{remark}
  The Teichmüller embedding is not in general unique, but two different
  embeddings differ by
  translations with respect to the two measured foliations
  \cite[Theorem 3.7]{FB18:Couch}.
\end{remark}

Ioffe's theorem gives a pair of measured foliations on $R$
and~$S$. To relate to Theorem~\ref{thm:emb}, we need to
approximate both of these measured foliations by simple multi-curves.
This is
more subtle than it appears at
first, since the natural map $f_* \co \Curves^+(R) \to
\Curves^+(S)$ does \emph{not} generally extend to a continuous map
$\MF^+(R) \to \MF^+(S)$, as the following example shows.

\begin{example}\label{examp:erase-hole}
  Let $S = S^2 \setminus \{D_a,D_b,D_c\}$ be the sphere minus three
  disks and let $R = S^2 \setminus \{D_a,D_b,D_c,D_d\}$ be the subsurface
  obtained by removing another disk. Pick a set of disjoint arcs
  $\gamma_{a,b}$, $\gamma_{a,c}$,
  $\gamma_{b,d}$, and $\gamma_{c,d}$ on~$S$ between the
  respective boundary components. For $s = p/q$ a positive rational number,
  there is a natural simple curve $C_s$ at slope $s$ with
  \begin{align*}
    i(\gamma_{a,c},C_s) &= i(\gamma_{b,d},F_s) = q\\
    i(\gamma_{a,b},C_s) &= i(\gamma_{c,d},F_s) = p,
  \end{align*}
  as illustrated in Figure~\ref{fig:sample-curves}.
\begin{figure}
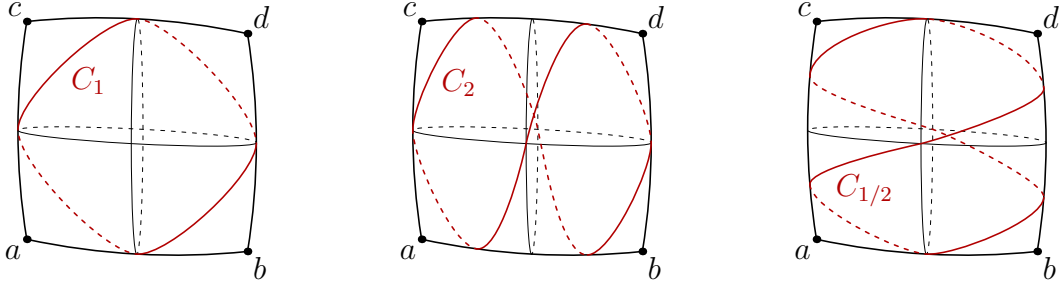

  \[
  \mfigb{pillow-1}
  \qquad\qquad
  \mfigb{pillow-2}
  \qquad\qquad
  \mfigb{pillow-3}
  \]
  \caption{Some of the curves $C_s$ in Example~\ref{examp:erase-hole}.}
  \label{fig:sample-curves}
\end{figure}
  Set $F_s \coloneqq (1/q) \cdot [C_s]$ for $s \in \QQ_+$, so that
  \begin{align*}
    i(\gamma_{a,c},F_s) &= i(\gamma_{b,d},F_s) = 1\\
    i(\gamma_{a,b},F_s) &= i(\gamma_{c,d},F_s) = s.
  \end{align*}
  Then $F_s$ extends to a continuous family of foliations for $s \in
  \RR_+$.

  For $s \in \QQ_+$, if we push forward $F_s$ by the inclusion map~$f$,
  we get a multiple
  of a simple curve
  on~$S$. There are only three simple curves on~$S$, the curves $C_a$,
  $C_b$, and $C_c$ around the respective boundary component. Which one
  we get depends only on the parity of $p$ and~$q$, where $s=p/q$ in
  lowest terms:
  \begin{equation}\label{eq:erase-hole}
    f_*[F_s] = \frac{1}{q} \cdot 
      \begin{cases}
        [C_a] & \text{$p$ odd, $q$ odd}\\
        [C_b] & \text{$p$ odd, $q$ even}\\
        [C_c] & \text{$p$ even, $q$ odd.}\\
      \end{cases}
  \end{equation}
  This map $f_*$
  has no continuous extension to~$\RR_+$.
\end{example}

\begin{example}
  We can improve Example~\ref{examp:erase-hole} to avoid dealing with
  curves around boundary components. Let $S'$ be the
  surface obtained from the previous surface~$S$ by gluing a pair of
  pants to $\partial D_a$, $\partial D_b$, and $\partial D_c$, and
  similarly glue a pair of pants to~$R$ to get~$R'$. Then $S'$ is a
  surface of genus two and $R'$ is a surface of genus two minus a
  disk. Then $F_s$ can be viewed as a continuous
  family of foliations on $R'$, and Equation~\eqref{eq:erase-hole}
  still holds.
\end{example}

Despite Example~\ref{examp:erase-hole}, we can still do simultaneous
approximations, using the techniques of Proposition~\ref{prop:curves-dense}.

\begin{proposition}\label{prop:slit-approx}
  Let $f \co R \to S$ be a topological slit map with respect to
  $F_S \in \MF^+(S)$. Let $F_R = f^* F_S$. Then there
  is a sequence of simple multi-curves $C_n$ on $R$ and weights
  $\lambda_n$ so that
  \begin{align*}
  \lim_{n \to \infty}\lambda_n F[C_n] &= F_R\\
  \lim_{n \to \infty}\lambda_n F[f(C_n)] &= F_S.
  \end{align*}
\end{proposition}

\begin{proof}
  By Lemma~\ref{lem:MF-carried}, $[F_R] = T_R(w)$ for some weight~$w$
  on a taut train
  track $T_R$ on~$R$. Fix a boundary component~$B$ of~$R$, and let
  $\beta$ be a curve parallel to~$B$ slightly pushed in to~$R$.
  If $f(\beta)$ bounds a disk in~$S$, the corresponding slit of~$F_S$
  is a tree which must have at least two endpoints. Each endpoint of
  the tree contributes a zero to $F_R$ on~$B$, so $F_R$ has at least
  two zeros on~$B$.

  Likewise,
  if $f(\beta)$ bounds a once-punctured disk in~$S$, the corresponding
  slit of~$F_S$ is a tree with at least two endpoints. At most one of
  these endpoints may be at the puncture, so $F_R$
  has at least one zero on~$B$.

  Let $T_S = f(T_R)$. The second part of the statement of
  Lemma~\ref{lem:MF-carried}
  guarantees that $T_S$ is taut, and so $F_S = T_S(w)$. (The new disks
  in~$T_S$ that were not
  disks in~$T_R$ have at least two cusps, and the new once-punctured disks
  have at least one cusp.)

  As in the proof of Proposition~\ref{prop:curves-dense},
  choose a sequence of rational weights $w_n \in
  \Meas_{\QQ}(T_R)$ approaching~$w$, and choose scalars~$\lambda_n$ so
  that $w_n' \coloneqq w_n/\lambda_n$ is integral. Then
  $T_R(w_n')$ is a multi-curve $[C_n]$ with $\lambda_n [C_n]$
  approaching $[F_R]$.
  We also have $[f(C_n)] = T_S(w_n')$, so by
  Lemma~\ref{lem:tt-continuous}, $\lambda_n [f(C_n)]$ approaches
  $F_S$.
\end{proof}

\begin{proof}[Proof of Theorem~\ref{thm:emb}]
  If $S = \bigcup_i S_i$ is not connected, with $R_i = f^{-1}(S_i)$,
  then the stretch factor is
  a supremum over all embeddings $R_i \hookrightarrow S_i$, as
  $\frac{a+b}{c+d} < \max\bigl(\frac{a}{c} + \frac{b}{d}\bigr)$. On
  the
  other hand, $R$
  conformally embeds in $S$ iff $R_i$ conformally
  embeds in $S_i$ in the given homotopy class for all~$i$. So from now
  on we assume that $S$ is
  connected.

  If $f\co R \hookrightarrow S$ is homotopic to a
  conformal embedding, then Proposition~\ref{prop:el-total-el}
  guarantees that for
  all multi-curves $[C] \in \Curves^+(R)$, we have $\EL_S[f(C)] \le
  \EL_R[C]$, as we have more maps in computing $\EL_S[f(C)]$, so
  smaller infimum in Equation~\eqref{eq:el-total-el}. Thus $\SF[f] \le 1$.

  Conversely, suppose $f$ is not homotopic to a conformal
  embedding. Then by Theorem~\ref{thm:ioffe}, 
  $f$ is homotopic to a Teichmüller map~$g$ of dilatation~$K$ with respect
  to quadratic
  differentials $q_R \in \Quad^+(R)$ and~$q_S \in \Quad^+(S)$.  Apply
  Proposition~\ref{prop:slit-approx} to find a sequence of simple
  multi-curves
  $C_n$ on~$R$ and weights~$\lambda_n$ so that $\lambda_n[C_n]$
  approximates $\Fh(q_R)$ and $\lambda_n[f_*C_n]$ approximates
  $\Fh(q_S)$. By Theorem~\ref{thm:heights}, the quadratic differentials
  corresponding to $\lambda_n[C_n]$ approach $q_R$ and the quadratic
  differentials corresponding to $\lambda_n [f(C_n)]$ approach
  $q_S$, and therefore
  \begin{equation}\label{eq:SF-bound}
  \SF[f] \ge 
  \lim_{n \to \infty} \frac{\EL_S[f(C_n)]}{\EL_R[C_n]} = 
  \frac{\EL_S[\Fh(q_S)]}{\EL_R[\Fh(q_R)]} = 
  \frac{\norm{q_S}}{\norm{q_R}} = K > 1.\qedhere
  \end{equation}
\end{proof}

When the stretch factor is larger than~$1$, we find it exactly
(Proposition~\ref{prop:sf-qc}) with the following standard fact.

\begin{lemma}\label{lem:qc-sf}
  Let $f \co R \hookrightarrow S$ be a quasi-conformal embedding of
  Riemann surfaces with quasi-conformal constant $\le K$, and let $C$
  be any multi-curve on~$R$. Then
  \[
  \EL_S[f(C)] \le K \EL_R[C].
  \]
\end{lemma}

\begin{proof}[Proof of Proposition~\ref{prop:sf-qc}]
  We can again assume that $S$ is connected. If $\SF[f] = 1$, the
  result is trivial: By Theorem~\ref{thm:emb}, there is a conformal
  embedding, which has quasi-conformal constant equal to~$1$.
  If $\SF[f] > 1$, then by Theorem~\ref{thm:emb}, the map~$f$ is not
  homotopic to a conformal embedding. Ioffe's
  Theorem~\ref{thm:ioffe} constructs a $K$-quasi-conformal
  map. $\SF[f] \le K$ by Lemma~\ref{lem:qc-sf}, and $\SF[f] \ge
  K$ by Equation~\eqref{eq:SF-bound}.
\end{proof}


\section{Strict embeddings}
\label{sec:strict}

We now turn to Theorem~\ref{thm:strict-emb}, on embeddings with stretch
factor strictly less than~$1$. We start with some preliminary lemmas.

\begin{lemma}\label{lem:area-bound}
  Let $f \co R \hookrightarrow S$ be a strict conformal
  embedding. Then there is a constant $K<1$ so that for any $q \in
  \Quad^+(S)$,
  \[
  A_q(f(R)) \le K A_q(S).
  \]
\end{lemma}

\begin{proof}
  For any non-zero quadratic differential~$q$ on~$S$, the ratio
  $A_q(f(R))/A_q(S)$ is less than~$1$, as the open set
  missed by the image of~$f$ has some non-zero area with respect
  to~$q$. Then
  $A_q(f(R))/A_q(S)$ is a continuous function on the projective space
  $P\Quad^+(S)$. Since
  $P\Quad^+(S)$ is compact, the result follows.
\end{proof}

Later, in Theorem~\ref{thm:area-surface}, we will strengthen
Lemma~\ref{lem:area-bound} considerably.

\begin{lemma}\label{lem:annular-sf}
  Let $R$ be a compact Riemann surface with a quadratic
  differential~$q\in\Quad^+(R)$ that is strictly positive on $\partial
  R$.  Let
  $\widehat{R}_t$ be the annular extension of~$R$ obtained by gluing a
  Euclidean cylinder of width~$t$ onto each boundary component of~$R$
  with respect to the locally Euclidean metric given by~$q$.  Then
  \[
  \lim_{t \to 0} \SF[\widehat{R}_t\to R] = 1,
  \]
where $\SF[\widehat{R}_t\to R]$ is the stretch factor of the
obvious homotopy class of topological embeddings.
\end{lemma}

\begin{proof}
  By Proposition~\ref{prop:sf-qc}, it suffices to construct a family of
  quasi-conformal maps
  $f_t\co \widehat{R}_t \to R$ with quasi-conformal constant $K_t$
  that approaches~$1$ as $t$ approaches~$0$.  The assumption that $q$
  is positive on $\partial
  R$ guarantees that near each component~$C_i$ of $\bdy R$ there is an
  annulus $A_i$ foliated by leaves of~$\Fh(q)$, with circumference
  $r_i$ and width $w_i$ (with
  respect to the Euclidean metric induced
  by~$q$).  Let $B_{i,t}$ be the
  annulus added to this boundary component in $\widehat{R}_t$, and
  let $\iota_t\co \widehat{R}_t \to R$ be the affine map of
  $A_i\cup B_{i,t}$ onto $A_i$
  and the identity outside of $A_i \cup B_{i,t}$. Then $\iota_t$
  has quasi-conformal constant equal to
  \[
  \max_i \frac{w_i + t}{w_i},
  \]
  which goes to~$1$ as $t \to 0$ as desired.
\end{proof}

\begin{proof}[Proof of Theorem~\ref{thm:strict-emb}]
  \textbf{(\ref{item:strict-annular}) $\Rightarrow$
    (\ref{item:strict-strict}):} An annular conformal embedding is
  also a strict conformal embedding, so this is clear.

  \textbf{(\ref{item:strict-strict}) $\Rightarrow$
    (\ref{item:strict-sf}):} Suppose that $f$ is a strict conformal
  embedding, and let $K < 1$ be the constant from
  Lemma~\ref{lem:area-bound}. For any multi-curve
  $[C]\in\Curves^+(R)$, let $q\in\Quad^+(S)$ be the quadratic
  differential that realizes extremal length for $[f(C)]$, and
  consider the pull-back metric $\mu = f^*\abs{q}$ on~$R$. 
  Since $(R, \mu)$ and $(f(R), \abs{q})$ are isometric, but there are
  more curves in the homotopy class~$[C]$ on~$S$ than those that lie in
  $f(R)$, we have
  $\ell_\mu[C] \ge \ell_{\abs{q}}[f(C)]$. Therefore,
  \begin{align*}
    \EL_R[C] &\ge \frac{\ell_\mu[C]^2}{A_\mu(R)} \ge
             \frac{\ell_{\abs{q}}[f(C)]^2}{KA_{\abs{q}}(S)}
             = K^{-1} \EL_S[f(C)]
  \end{align*}
  implying that $\frac{\EL[f(C)]}{\EL[C]} \le K$. Since $C$ was
  arbitrary, $\SF[f] \le K$.

  \textbf{(\ref{item:strict-sf}) $\Rightarrow$
    (\ref{item:strict-annular}):} Suppose that $\SF[f] < 1$.  Pick a
  quadratic differential~$q\in\Quad^+(R)$ that is real and strictly positive on
  $\partial R$.  Let $\widehat{R}_t$ be the family of annular
  extensions of~$R$ with respect to~$q$ as in Lemma~\ref{lem:annular-sf}, and
  let $\widehat{f}_t:
  \widehat{R}_t \to S$ be the composite topological embeddings.
  Then by Proposition~\ref{prop:sf-compose},
  \[
  \SF[\widehat{f}_t] \le \SF[\widehat{R}_t\to R] \cdot \SF[f].
  \]
  It follows from Lemma~\ref{lem:annular-sf} that for $t$
  sufficiently small, $\SF[\widehat{f}_t] \le 1$, so by
  Theorem~\ref{thm:emb}, $\widehat{f}_t$ is homotopic to a
  conformal embedding.

  \textbf{(\ref{item:strict-sf}) $\Leftrightarrow$
    (\ref{item:strict-ball}):} This is a consequence of
  Proposition~\ref{prop:sf-dist}, which we prove next.
\end{proof}

\begin{proof}[Proof of Proposition~\ref{prop:sf-dist}]
  By compactness of balls in Teichmüller space, it suffices to show,
  on one hand, that if $d(S,S') < -\frac{1}{2}\log\SF[f]$, then there
  is a conformal embedding of $R$ in~$S'$; and, on the other hand,
  that there are surfaces $S'$ with $d(S,S')$ arbitrarily close to
  $-\frac{1}{2}\log\SF[f]$ so that $R$ does not conformally embed
  in~$S'$.
  
  For the first part, suppose
  $d(S, S') < -\frac{1}{2}\log \SF[f]$. Let $\id_{S,S'}$ be the
  identity map from the marking. Then
  \begin{align*}
    \SF[\id_{S,S'} \circ f] \le \SF[f] \cdot \SF[\id_{S,S'}] 
                       \le \SF[f] \cdot \SF^{\pm}[\id_{S,S'}] 
                      = \SF[f] \cdot e^{2d(S,S')}
                       < 1.
  \end{align*}
  as desired.

  To get the other direction of the inequality, pick $\epsilon > 0$,
  and set $K = e^{\epsilon}/\SF[f]$ and $\lambda = \frac{K-1}{K+1}$.
  Find a simple multi-curve~$C$ on~$R$ near the
  supremum defining $\SF$:
  \[
    \frac{\EL_S[f(C)]}{\EL_R[C]} > e^{-\epsilon} \SF[f].
  \]
  Let $q = q_{f(C)}\in \Quad^+(S)$ be the associated Jenkins-Strebel
  quadratic differential, and set
  $\mu = \lambda \cdot {\overline{q}}/{\abs{q}}$ to be an associated
  Beltrami differential. Let $S'$ be $S$ stretched by~$\mu$, so that
  \[
    d(S,S') = \frac{1}{2}\log \SF^\pm[\id_{S,S'}]
    = \frac{1}{2}\log \SF[\id_{S,S'}]
    = \frac{1}{2}\frac{\EL_{S'}[f(C)]}{\EL_S[f(C)]}
    = \frac{\log K}{2}
    = -\frac{1}{2}\log \SF[f] + \frac{\epsilon}{2}.
  \]
  We also have
   \[
     \SF[\id_{S,S'} \circ f]
     \ge \frac{\EL_{S'}[f(C)]}{\EL_R[C]}
     > \bigl(e^{\epsilon}/\SF[f]\bigr)\bigl(e^{-\epsilon}\SF[f]\bigr)
     > 1
   \]
 so $S' \notin \Teich_R(S)$.
 Since $\epsilon$ can be chosen arbitrarily small, we get the
 desired result.
\end{proof}

\begin{remark}
  It follows from the proof that the
  stretching to a nearest point on $\partial\Teich_R(S)$ is
  horizontal on the boundary.
\end{remark}

\begin{remark}
  The nearest point to~$S$ on $\partial\Teich_R(S)$ is not always
  unique, as we can see from the fact that $\wt\SF[f] \ne \SF[f]$ in
  Examples~\ref{examp:cover} and~\ref{examp:cover2} below.
  Indeed, let $f \co R \hookrightarrow S$ be a conformal embedding and
  let $\wt f \co \wt R \hookrightarrow \wt S$ be a regular covering
  with $\SF[f] < \SF[\wt f] < 1$. Then if there were a unique nearest
  point~$\wt S'$ to $\wt S$ on $\partial \Teich_{\wt R}(\wt S)$, it
  would be invariant under the deck transformations, and so would
  descend to give a point~$S'$ on $\partial\Teich_R(S)$, contradicting
  $\SF[f] < \SF[\wt f]$.
\end{remark}


\section{Behavior under finite covers}
\label{sec:covers}

We now turn to the behavior of the stretch factor under finite covers. We
start with some easy statements.

\begin{lemma}\label{lem:el-cover}
  Let $\pi \co \wt{S} \to S$ be a covering map of Riemann surfaces of
  finite degree~$d$. For $C$
  a weighted multi-curve on~$S$, define $\pi^{-1} C$ to be the
  full inverse image of~$C$, with the same weights. Then
  $\EL_{\wt{S}}[\pi^{-1} C] = d\EL_S[C]$.
\end{lemma}

\begin{proof}
  By Proposition~\ref{prop:el-area},
  $\EL_S[C] = A_{q_C}(S)$, where $q_C$ is the Jenkins-Strebel
  quadratic differential corresponding to~$C$. Then
  $f^*(q_C)$ is a Jenkins-Strebel quadratic differential
  corresponding to $\pi^{-1}(C)$, and so
  \[
  \EL_{\wt{S}}[\pi^{-1}(C)] = A_{f^*(q_C)}(S) = d A_{q_C}(S) = d \EL_S[C].\qedhere
  \]
\end{proof}

\begin{lemma}\label{lem:SF-cover-inc}
  For $\wt{f}$ a finite cover of $f\co R \hookrightarrow S$, we have
  $\SF[\wt{f}] \ge \SF[f]$.
\end{lemma}

\begin{proof}
  Follows from Lemma~\ref{lem:el-cover} and the definition of $\SF$,
  as the supremum involved in computing $\SF[\wt{f}]$ is over a larger
  set.
\end{proof}

\begin{proposition}\label{prop:SF-large-cover}
  If $f \co R \hookrightarrow S$ is a topological embedding of Riemann
  surfaces with $\SF[f] \ge 1$ and $\wt f$ is a finite cover of~$f$ in the
  sense of Definition~\ref{def:cover-map}, then $\SF[\wt f] = \SF[f]$.
\end{proposition}

\begin{proof}
  If $\SF[f] = 1$, the result follows from
  Lemma~\ref{lem:SF-cover-inc} and Theorem~\ref{thm:emb}.

  If $\SF[f] > 1$,
  by Proposition~\ref{prop:sf-qc} $\SF[f]$ is the minimal
  quasi-conformal constant of any map homotopic to~$f$, which by
  Theorem~\ref{thm:ioffe} is given by a Teichmüller embedding~$g$. Let
  $\wt{g}$ be the corresponding
  cover of~$g$. Then $\wt{g}$ is also a Teichmüller embedding with the same
  quasi-conformal constant, and so $\SF[\wt{f}]$ is the
  quasi-conformal constant of~$\wt{g}$.
\end{proof}

\begin{remark}
  Proposition~\ref{prop:SF-large-cover} relies on $\wt{f}$ being a cover
  of finite degree of~$f$. McMullen \cite[Corollary 1.2]{McMullen89:Amenable}
  shows that, in the case that $R$ and $S$ are closed surfaces, $f$
  is a Teichmüller map, and $\wt{f}$ is a non-amenable cover of~$f$,
  then $\wt{f}$ does \emph{not} minimize the quasi-conformal
  distortion in its bounded homotopy class.
\end{remark}

\begin{proposition}\label{prop:sf-cover-1}
  For $\wt{f}$ a finite cover of $f\co R \hookrightarrow S$, 
  the quantity $\SF[\wt{f}]$ is less than one, equal to one, or
  greater than one exactly when $\SF[f]$ is less than one, equal to one, or
  greater than one.
\end{proposition}

\begin{proof}
  If $\SF[f] < 1$, by Theorem~\ref{thm:strict-emb}, $f$
  is homotopic to a
  strict conformal embedding. Since a cover of a strict conformal
  embedding is a strict conformal embedding, we have $\SF[\wt{f}] < 1$.
  The other cases follow from Proposition~\ref{prop:SF-large-cover}.
\end{proof}

Although there is some good behavior, it is not true in general that
$\SF[\wt{f}]=\SF[f]$.

\begin{example}\label{examp:cover}
  Let $R$ and~$S$ both be disks with two
  points removed, with $f \co R \to S$ a strict conformal embedding
  and $g \co S \to R$ a homotopy inverse. The surfaces~$R$ and~$S$
  have, up to equivalence and scale, only one non-trivial
  simple multi-curve (the boundary-parallel curve), so
  $\SF[f] = 1/\SF[g]$. Also, $SF[f] < 1$, since $f$ was assumed to be
  a strict conformal embedding. Now take any non-trivial finite
  cover~$\wt R$ of~$R$ and the corresponding cover~$\wt S$ of~$S$. Let
  the corresponding topological
  embeddings be $\wt{f} \co \wt{R}\to\wt{S}$ and $\wt g \co \wt S \to \wt
  R$. Since $\SF[g] > 1$, by Proposition~\ref{prop:SF-large-cover} we
  have $\SF[\wt{g}] = \SF[g]$, with the supremum in the definition of
  stretch factor realized by a symmetric multi-curve. By
  Theorem~\ref{thm:ioffe}, the
  quadratic differentials realizing this stretch factor are
  \emph{unique}, so
  for \emph{any} non-symmetric multi-curve~$C$ on~$\wt S$ (or equivalently
  $\wt R$), we have
  \[
  \SF[g] = \SF[\wt{g}] > \frac{\EL_{\wt R}[C]}{\EL_{\wt S}[C]}.
  \]
  But then
  \[
  \SF[\wt{f}] \ge \frac{\EL_{\wt S}[C]}{\EL_{\wt R}[C]} > 1/\SF[g] = \SF[f].
  \]
\end{example}

\begin{example}\label{examp:cover2}
  The previous example can be improved to give an examples with
  arbitrarily large gap between $\SF$ and $\wt\SF$: for any
  $0 < \eps < \delta < 1$, there is an embedding
  $f \co R \hookrightarrow S$ and two-fold cover $\wt f$ so that
  $\SF[f] < \epsilon$ and $\SF[\wt f] > \delta$. This example is due
  to Maxime Fortier-Bourque. Let $R_t$ be the disk with two punctures
  obtained by doubling a $t \times 1$ rectangle along three of its
  sides, and let $S_t$ be
  the double cover of~$R_t$ branched along one of the two punctures.
  Then for $s < t$ the embedding $S_s \hookrightarrow S_t$ is a cover
  of the embedding $R_s \hookrightarrow R_t$.

  Let $C_1$ be the only non-trivial curve on $R_t$, the curve parallel
  to the boundary as shown on the
  left of Figure~\ref{fig:cover2}. 
  Let $C_2$ be the non-symmetric curve on~$S_t$
  shown on the right of Figure~\ref{fig:cover2}. By construction,
  $\EL_{R_t}[C_1] = 2/t$.
  As $t \to \infty$,
  the surface $S_t$ approaches a sphere with 4~punctures, specifically
  the double of a square. The curve $C_2$ is non-trivial on the
  4-punctured sphere, and so its extremal length approaches a definite
  value:
  \[
    \lim_{t \to \infty} \EL_{S_t}[C_2] = 2.
  \]
  Thus, for $t \gg s \gg 0$, we have
  \begin{align*}
    \SF[R_s \hookrightarrow R_t] &= \frac{2/t}{2/s} = \frac{s}{t}\\
    \SF[S_s \hookrightarrow S_t] &\ge \frac{\EL_{S_t}[C_2]}{\EL_{S_s}[C_2]}
       \rightarrow 1,
  \end{align*}
  as desired.
  
  With a little more care, one can show that
    $\EL_{S_t}[C_2] \approx 2(1 + Ke^{\pi t/2})$ for some
    constant~$K$. This uses the uniformization of $S_\infty$ to the double
    of a square by the composition of $z \mapsto \sin(\pi i z/2)$ and
    $z \mapsto \int_{w=0}^z dw/\sqrt{w^3-w}$.
\end{example}
\begin{figure}
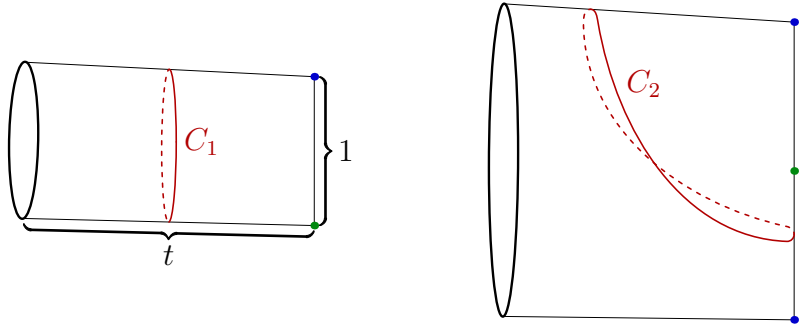

  \[
    \mfigb{cover-examp-0}
    \qquad\qquad
    \mfigb{cover-examp-1}
  \]
  \caption{The surfaces from Example~\ref{examp:cover2}. Left: The
    family of surfaces $R_t$, with the unique non-trivial curve~$C_1$.
    Right: the double cover $S_t$, with the curve~$C_2$.}
  \label{fig:cover2}
\end{figure}

In order to prove Theorem~\ref{thm:sf-cover}, we need some extra
control: a strengthening of Lemma~\ref{lem:area-bound}.

\begin{theorem}\label{thm:area-surface}
  Let $f\co R \hookrightarrow S$ be a annular conformal
  embedding of Riemann surfaces. Then there is a
  constant $K < 1$ so that for any quadratic differential $q \in
  \Quad(S)$,
  \[
  A_{f^* q}(R) \le K A_q(S).
  \]
  Furthermore, the constant $K$ can be chosen uniformly under
  finite covers, in the sense that for any finite cover
  $\wt{f}\co \wt{R}\to\wt{S}$
  of~$f$ and any
  quadratic differential $\wt{q} \in \Quad(\wt{S})$,
  \[
  A_{f^* \wt{q}}(\wt{R}) \le K A_{\wt{q}}(\wt{S}).
  \]
\end{theorem}

The technique in Lemma~\ref{lem:area-bound} will not work to prove
Theorem~\ref{thm:area-surface}, as
$\Quad(S)$ is infinite-dimensional. (That bound is also not uniform
under covers.)
As in Lemma~\ref{lem:area-bound}, $K$ depends on the actual embedding,
not just the
homotopy class of the embedding.

When $S$ is a disk, Theorem~\ref{thm:area-surface} is not hard.
For $a \in \mathbb{C}$ and $r>0$, we denote by
$\D(a,r)=\{z:\abs{z-a}<r\}$ the open disk of radius $r$ about $a$.

\begin{proposition}\label{prop:small-area}
  Let $\Omega \subset \DD$ be an open subset of the disk so that
  $\overline{\Omega} \cap \partial\DD = \emptyset$. For any
  quadratic differential $q\in\Quad(\DD)$,
  \[
  A_q(\Omega) \le r^2 A_q(\DD),
  \]
  where $r$ is large enough so that $\Omega \subset \D(0,r)$.
\end{proposition}

Proposition~\ref{prop:small-area} is a special case of
Proposition~\ref{prop:disk-area} below, but we give a separate proof
because we can give a precise constant.

\begin{proof}
  Let $r_0$ be the smallest value so that
  $\Omega \subset \D(0,r_0) \subset \DD$, and let $q \in \Quad(\DD)$
  be arbitrary. For $0 \le r \le 1$, we will show that
  $A_q(\D(0,r)) \le r^2 \cdot A_q(\DD)$, so that $K = r_0^2$
  suffices. Define
  \begin{align*}
    I(r) &= \int_{\theta=0}^{2\pi} \abs{q(r e^{i\theta})} \, d\theta\\
    J(r) &= \int_{s=0}^r s I(s)\,ds = A_q(\D(0,r)),
  \end{align*}
  where we are writing $q = q(z)\,(dz)^2$ with $q(z)$ a holomorphic
  function. The function $z\mapsto\abs{q(z)}$ is subharmonic, so if $s
  < r$, we have $I(s) \le I(r)$. (We would have equality between the
  corresponding integrals if $\abs{q(z)}$ were harmonic; see, e.g.,
  \cite[p.\ 142]{Burckel79:IntroComplex}). We therefore
  have $J(r) = \int_{s=0}^r s I(s)\,ds \le r^2 I(r)/2$, and so
  \[
  \frac{d}{dr} \frac{J(r)}{r^2}
    = \frac{r J'(r) - 2J(r)}{r^3}
    \ge \frac{r^2 I(r) - r^2 I(r)}{r^3} = 0.
  \]
  It follows that $J(r)/r^2 \le J(1)$, as desired.
\end{proof}

Proposition~\ref{prop:small-area} is false if $\overline{\Omega}$ is
allowed to intersect $\partial\DD$. Suppose $\Omega$ contains a
neighborhood of a segment of $\partial\DD$, and let $w$ be a point
very close to this segment.  By a conformal automorphism $\phi$ of $\DD$, we
can take $w$ to the center of the disk. Then
$\bigl(d\phi(z)\bigr)^2$ will have its measure concentrated near
$w\in\Omega$, as illustrated in
Figure~\ref{fig:disk-trans}.

\begin{figure}
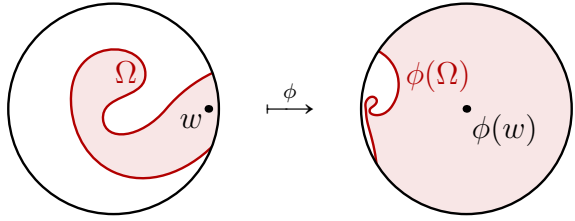

  \[
  \mfigb{embedding-20} \quad\overset{\phi}{\longmapsto}\quad \mfigb{embedding-21}
  \]
  \caption{Möbius transformations to make the area of a quadratic
    differential be concentrated near a point $w$ that is close to
    $\partial\DD$.}
\label{fig:disk-trans}
\end{figure}

The following proposition says that this is all that can happen: if
the mass of $q$ on $\Omega$ gets large, then the mass of $q$ is
concentrating near $\partial\DD$.

\begin{proposition}\label{prop:disk-area}
Let $\Omega \subset \D$ be an open subset of the disk with an open
set~$A$
in its complement, and let $B \subset \overline{\D}$ be a neighborhood
of $\overline{\Omega} \cap \partial \overline{\D}$, as illustrated in
Figure~\ref{fig:surface-decomp}. Then, for every
$\epsilon>0$, there is a $\delta>0$ so that if $q \in \Quad(\D)$
is such that $q \ne 0$ and
\begin{align*}
\frac{A_q(\Omega)}{A_q(\D)} &> 1-\delta,
\shortintertext{then}
\frac{A_q(B)}{A_q(\D)} &> 1-\epsilon.
\end{align*}
\end{proposition}
\begin{figure}
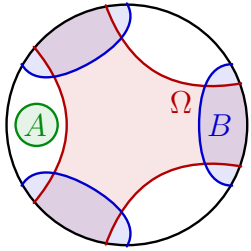

  \[\mfigb{embedding-10}\]
  \caption{The schematic setup of Proposition~\ref{prop:disk-area}.}
\end{figure}

The proposition implies that given a sequence $q_n \in
\mathcal{Q}(\D)$, if the percentage of the $\abs{q_n}$-area of $\D$
occupied by $\Omega$ tends to $1$, then the percentage of the
$\abs{q_n}$-area occupied by the set $B$ of ``thickened ends of $\Omega$''
also tends to $1$. Figure~\ref{fig:disk-trans} again provides an
example of how this happens.

We give two versions of the proof, one shorter, and the other more
explicit and giving (poor) bounds on the constants.

\begin{proof}[Proof of Proposition~\ref{prop:disk-area}, version 1]
  If there are no such bounds as in the statement of the proposition,
  there is an $0 < \epsilon < 1$ and
  a sequence of quadratic differentials $q_n\in\Quad(\DD)$ so that
  \begin{align}
    A_{q_n}(\DD) &= 1\\
    A_{q_n}(B) &< 1-\epsilon\label{eq:B-small}\\
    A_{q_n}(\Omega) &> 1-1/n.\label{eq:omega-large}
  \end{align}
  Consider $A_{q_n}$ as a measure on $\overline{\DD}$. Since the
  space of measures of unit area on the closed disk is compact in the
  weak topology, after passing to a subsequence we may assume that
  $A_{q_n}$ converges (weakly) to some limiting measure
  $\mu$ (of total mass~$1$) on~$\overline{\DD}$. Since holomorphic
  functions on the disk
  that are also in $L^1(\DD)$ form a normal family,
  after passing to a further subsequence, we may assume that the sequence~$q_n$
  converges locally uniformly to some holomorphic function~$q_\infty$
  on~$\DD$. The
  restriction of~$\mu$ to the open disk is then~$A_{q_\infty}$. But
  $A_{q_n}(A) < 1/n$, so
  $A_{q_\infty}(A) = 0$, so $q_\infty$ is identically~$0$ on~$A$ and
  therefore on the entire open disk. Hence
  $\mu$ is supported on $\bdy\oDD$. Equation~\eqref{eq:omega-large}
  implies that the support of $\mu$ is also contained in
  $\overline{\Omega}$, and hence in
  $\overline{\Omega}\cap\bdy\oDD$. But this contradicts
  Equation~\eqref{eq:B-small}.
\end{proof}

\begin{proof}[Proof of Proposition~\ref{prop:disk-area}, version 2]
Apply a Möbius transformation so that $A$ contains~$0$. We may then
assume that $\Omega \subset \D \setminus \overline{\D(0,2r_0)}$ for
some $0 < r_0 < 1/2$.
We identify the space $\mathcal{Q}(\D)$ of integrable holomorphic
quadratic differentials on $\D$ with the Banach space of
$L^1$-integrable holomorphic functions on $\D$, so that
$A_q(\D)=\int_\D\abs{q}=\norm{q}$.

Suppose $q \in \mathcal{Q}(\D)$ satisfies $A_q(\D)=1$.  We will
quantitatively show that the $q$-area of a small ball controls the
$q$-area of a big ball.
Suppose $s$ is chosen close to~$1$ with $r_0<s<1$. Suppose
$\abs{z}\leq s$.
The Cauchy Integral Formula applied to the
concentric circles comprising the disk $\D(z,1-s)$ shows that
\[ \abs{q(z)} \leq \frac{1}{\pi(1-s)^2}\int_{\D(z,1-s)}\abs{q}=\frac{1}{\pi(1-s)^2}A_q(\D(z, 1-s)),\]
i.e., $\abs{q}$ is subharmonic.
Using the assumption that $A_q(\D)=1$, this implies 
\begin{align}
\label{eqn:cif1}
\abs{z} \leq s &\implies \abs{q(z)} \le K(s):=\frac{1}{\pi(1-s)^2}.\\
\intertext{Similar reasoning shows}
\label{eqn:cif2}
\abs{z} \leq r_0 &\implies \abs{q(z)} \leq \frac{1}{\pi r_0^2} A_q(\D(0,2r_0)).
\end{align}
For $0<t<1$, let $M_q(t)$ be $\max\{\abs{q(z)} : \abs{z}=t\}$.  The
Hadamard Three Circles Theorem \cite[Theorem
6.3.13]{Conway78:Complex} implies that $\log M_q$ is a convex
function of $\log t$. Thus if $r$ and $r_1$ are chosen so that $r_0
\leq r \leq r_1 < s$ then
\begin{align*} 
\log M_q(r) &\le \log M_q(r_0)+\frac{\log M_q(s)-\log M_q(r_0)}{\log s-\log r_0}(\log r-\log r_0) \\
& \le \log M_q(r_0)+\frac{\log K(s) - \log M_q(r_0)}{\log s-\log r_0}(\log r_1-\log r_0)\\
& = \left(1-\frac{\log r_1 - \log r_0}{\log s- \log r_0}\right) \log M_q(r_0)+ \log K(s) \frac{\log r_1 - \log r_0}{\log s- \log r_0}\\
& = K_1 \log M_q(r_0) + K_2
\end{align*}
where $K_1$ and $K_2$ are constants, with $K_1>0$, depending only on
$r_0$, $r_1$, and~$s$, and not on~$q$.  It follows from (\ref{eqn:cif2}) that
there are positive constants $c_1$ and $c_2$ depending only on $r_0$,
$r_1$, and $s$ with
\begin{equation}
\label{eqn:compare_area}
A_q(\D(0,r_1)) < c_2 A_q(\D(0,2r_0))^{c_1}.
\end{equation}

Now suppose that $\delta$ is small, $0<\delta<1$, and
$A_q(\D\setminus\Omega)<\delta$.  Note that this implies that
$A_q(\D(0,2r_0))<\delta$.  Given $0<r_1<1$, let $E$ be the annulus
$\D\setminus\D(0,r_1)$.  From the definition of $B$, there is some
$r_1$ with $0<r_1<1$ close to $1$ for which $E\cap\Omega \subset E\cap
B$.  Choose $s$ so that $r_0<r_1<s<1$; we are in the setup of the
previous paragraph. We have 
\begin{align*}
1-c_2\delta^{c_1} & < A_q(E) && \mbox{by (\ref{eqn:compare_area})} \\
& = A_q(E\cap(\D\setminus \Omega)) + A_q(E \cap \Omega) \\
& < A_q(\D \setminus \Omega) + A_q(E\cap B) \\ 
& < \delta + A_q(B)
\end{align*}
and so $A_q(B)>1-c_2\delta^{c_1} - \delta$, which tends to $1$ as
$\delta$ tends to~$0$, as required.
\end{proof}

We also need an analogue of Proposition~\ref{prop:disk-area} for the
once-punctured disk. (In fact it is true
in more generality.)

\begin{proposition}\label{prop:punct-disk-area}
  Let $\DD^\times$ be the punctured unit disk $\DD \setminus \{0\}$,
  let $\Omega \subset
  \DD^\times$ be an open subset with an open set~$A$ in its
  complement, and let $B \subset \overline{\DD}^\times$ be an open
  neighborhood of $\overline{\Omega}
  \cap \partial\overline{\DD}$.
 Then, for every
$\epsilon>0$, there is a $\delta>0$ so that if $q \in \mathcal{Q}(\D^\times)$
is such that $q \ne 0$ and
\begin{align*}
\frac{A_q(\Omega)}{A_q(\D)} &> 1-\delta,
\shortintertext{then}
\frac{A_q(B)}{A_q(\D)} &> 1-\epsilon.
\end{align*}
\end{proposition}

\begin{proof}
  Let $s \co \D \to \D$ be the squaring map $s(z) = z^2$. We can apply
  Proposition~\ref{prop:disk-area} to the tuple
  $(s^{-1}(\Omega), s^{-1}(A), s^{-1}(B))$. For every quadratic
  differential $q
  \in \Quad(\D^\times)$ with at most a simple pole at~$0$,
  $s^* q$ is a quadratic differential on~$\D^\times$ with no pole,
  and can thus be considered as a quadratic differential on~$\D$.
  Since for any $X \subset \D^\times$,
  \[
  A_{s^*q}(s^{-1}(X)) = 2A_q(X),
  \]
  the area bounds for $s^* q$ on $s^{-1}(\Omega)$ and $s^{-1}(B)$
  imply the same bounds for $q$ on $\Omega$ and~$B$, as desired.
\end{proof}

\begin{proof}[Proof of Theorem~\ref{thm:area-surface}]
  For simplicity, if $S$ has no boundary or has non-negative
  Euler characteristic, remove disks from
  $S\setminus R$ until it has boundary and negative
  Euler characteristic. Then enlarge $R$ until it is equal to
  $S$ minus an $\epsilon$-neighborhood of $\partial S$,
  and think about $R$ as a subset of~$S$.

  Now choose a maximal set of simple, non-intersecting and
  non-parallel arcs $\{\gamma_i\}_{i=1}^{k}$ on~$S$. These will divide
  $S$ into a collection of half-pants (i.e., hexagons) and once-punctured
  bigons; arrange the arcs so that they divide~$R$ in the same way,
  as illustrated in Figure~\ref{fig:surface-decomp}. Let
  $\{P_j\}_{j=1}^{\ell}$ be the connected components of
  $S \setminus \bigcup \gamma_i$, and let $G_i$ be small disjoint
  tubular
  neighborhoods of the $\gamma_i$ inside $S$. Let $P_j' = P_j \cap
  R$ and let $G = \bigcup_i G_i$. As detailed below, we can apply
  Propositions~\ref{prop:disk-area} or~\ref{prop:punct-disk-area} to
  each triple
  $\bigl(P_j, P_j', P_j \cap G\bigr)$ to show that if the area
  of a sequence of quadratic differentials~$q_n$ on~$S$ is concentrating
  within~$R$, then
  it is actually concentrating within $G$.

  \begin{figure}
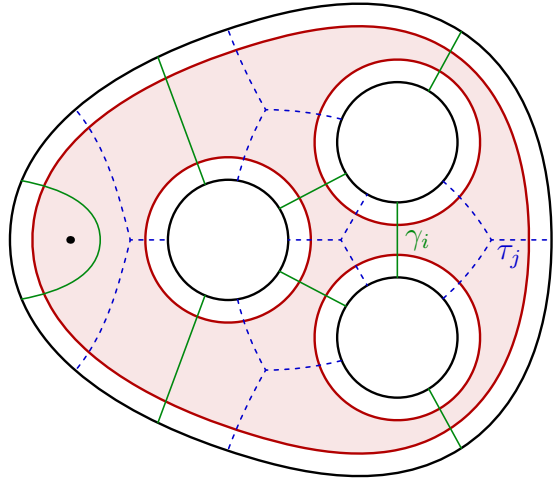

    \[
    \mfigb{embedding-5}
    \]
    \caption{Two decompositions of $S$ and~$R$. In this
      example, $S$ is a sphere with 4 holes and one puncture and
      $R$ is a smaller copy of~$S$ shaded in red. The
      arcs $\gamma_i$ (solid, in green) divide the two surfaces into
      half-pants and a once-punctured bigon. The tripods $\tau_j$ (dashed,
      in blue) divide the
      two surfaces into rectangles and a once-punctured bigon.}
    \label{fig:surface-decomp}
  \end{figure}

  We also pick another decomposition of $R$ and~$S$ into
  disks. Within each half-pants among the~$P_j$, pick a tripod
  $\tau_j$ with
  ends on the three components of $P_j \cap \partial S$ and
  intersecting $\partial R$ in three points, as
  in Figure~\ref{fig:surface-decomp}; ensure
  that $\tau_j$ is disjoint from $\overline{G}$. Let
  $\{Q_i\}_{i=1}^{k}$ be the connected components of
  $S \setminus \bigcup_j \tau_j$. Each $Q_i$ is a rectangle or a
  once-punctured bigon. Pick a
  small tubular neighborhood $T_j$ of
  $\tau_j$, small enough that each $T_j$ and $G_i$ are
  disjoint. Let $Q_i' = Q_i \cap R$ and $T = \bigcup_j T_j$.
  Propositions~\ref{prop:disk-area} and~\ref{prop:punct-disk-area}
  will again show that if
  the area of a sequence of quadratic differentials on~$S$ is
  concentrating
  within~$R$, then it is concentrating within $T$; but this
  is a contradiction, as $G$ and $T$ are
  disjoint.

  We now give the concrete estimates alluded to above. Since all areas
  are with respect to an arbitrary quadratic differential~$q \in
  \Quad(S)$, we will omit
  it from the notation for brevity. For
  each~$j$, the triple $(\overline{P_j}, P_j', G \cap P_j)$ is either a
  triple like $(\DD, \Omega, B)$ as in the statement of
  Proposition~\ref{prop:disk-area} or a triple like $(\DD^\times,
  \Omega, B)$ as in the statement of
  Proposition~\ref{prop:punct-disk-area}. We can thus find
  $\delta_j$ according to the
  propositions so that if $A(P_j') > (1-\delta)A(P_j)$,
  then $A(G \cap P_j) > (3/4)A(P_j)$. Let $\delta \coloneqq \min_j
  \delta_j$ and $\delta' \coloneqq \delta/4$.

  \begin{claim}
    If $A(R) > (1-\delta')A(S)$, then $A(G) > \frac{1}{2}\cdot A(S)$.
  \end{claim}

  \begin{proof}
    Let $J \subset \{1,\dots,\ell\}$ be the subset of indices $j$ so
    that $A(P_j') > (1-\delta)A(P_j)$, and let
    \begin{align*}
      P_J &\coloneqq \bigcup_{j \in J} P_j & P_J' &\coloneqq P_J \cap R\\
      P_\oJ &\coloneqq \bigcup_{j \notin J} P_j & P_\oJ' &\coloneqq P_\oJ \cap R.
    \end{align*}
    Then we have
    \begin{equation*}
      (1-\delta')A(S) - A(P_J') < A(P_\oJ')
        \le (1-\delta)A(P_\oJ)
         <  (1-\delta)(A(S) - A(P_J'))
    \end{equation*}
    which simplifies to
    \begin{equation*}
        A(P_J') > \frac{\delta - \delta'}{\delta}A(S) =
                  \frac{3}{4}\cdot A(S).
    \end{equation*}
    On the other hand, by the choice
    of~$\delta$, we have $A(G \cap P_J') > (3/4)A(P_J)$,
    so
    \begin{equation*}
      A(G) \ge A(G \cap P_J')
        > \frac{3}{4}\cdot A(P_J)
        \ge \frac{3}{4}\cdot A(P_J')
        > \frac{3}{4}\cdot \frac{3}{4} \cdot A(S)
        > \frac{1}{2} \cdot A(S). \qedhere
    \end{equation*}
  \end{proof}

  An exactly parallel argument shows that there is a $\delta'' > 0$ so
  that if $A(R) > (1 - \delta'')A(S)$, then $A(T) >
  \frac{1}{2}\cdot A(S)$. Since $G \cap T = \emptyset$, this implies that
  $A(R) \le (1-\min(\delta',\delta'))A(S)$, proving the first
  statement of the theorem.

  Note that the crucial constants $\delta'$ and $\delta''$ were
  defined as a minimum over the triples
  $(P_j,P_j',G \cap P_j)$ and $(Q_i, Q_i', T \cap Q_i)$. On a
  finite cover $\wt{f} \co \wt{R} \hookrightarrow \wt{S}$ of $f$,
  we can take arcs $\wt\gamma_i$ and tripods $\wt\tau_j$ to be lifts
  of $\gamma_i$ and $\tau_j$, respectively. Then the triples on the
  $\wt{S}$ are lifts of the triples on $S$, and the same
  estimate works in~$\wt{f}$.
\end{proof}

\begin{proof}[Proof of Theorem~\ref{thm:sf-cover}]
  If $\SF[f] \ge 1$, we have already proved the result in
  Proposition~\ref{prop:SF-large-cover}. If $\SF[f] < 1$, by
  Theorem~\ref{thm:strict-emb} we may assume that $f$ is an annular conformal
  embedding. Let $K$ be the constant from
  Theorem~\ref{thm:area-surface} for the map~$f$. We must show that
  for any finite cover $\wt{f}\co \wt{R} \to \wt{S}$ of~$f$ and any simple
  multi-curve~$\wt{C}$ on~$\wt{R}$,
  \[
  \frac{\EL_{\wt{R}}[\wt{f}(\wt{C})]}{\EL_{\wt{S}}[\wt{C}]} < K.
  \]
  Let $\wt{q}$ be the quadratic differential realizing the extremal
  length of $[\wt{f}(\wt{C})]$. Then, as in the proof of
  Theorem~\ref{thm:strict-emb},
  \[
  \EL_{\wt{R}}[\wt{C}]
    \ge\frac{\ell_{\wt{f}^*\abs{\wt{q}}}[\wt{C}]^2}{A_{\wt{f}^*\abs{\wt{q}}}(\wt{R})}
    \ge\frac{\ell_{\abs{\wt{q}}}[\wt{f}(\wt{C})]^2}{K A_{\abs{\wt{q}}}(\wt{S})}
    = K^{-1} \EL_{\wt{S}}[\wt{f}(\wt{C})]. \qedhere
  \]
\end{proof}


\section{Future challenges}\label{sec:challenges}

There are several obvious questions raised by Theorems~\ref{thm:emb},
\ref{thm:strict-emb}, and~\ref{thm:sf-cover}.
The first is an analogue of Proposition~\ref{prop:SF-large-cover} when
$\SF[f] < 1$.

\begin{problem}\label{prob:inf-SF}
  Give an intrinsic characterization of $\wt\SF[f]$ for general maps
  $f \co R \to S$ between Riemann surfaces as an infimum, not
  just when $\wt\SF[f] \ge 1$.
\end{problem}

To elaborate a little, $\SF$ and $\wt\SF$ are defined as maxima. It
would be much easier to find upper bounds (as in the hard direction of
Theorem~\ref{thm:sf-cover}) if there were an alternate definition of
$\wt\SF$ as a minimum. 
For example, there are two characterizations of extremal length: as a
maximum over metrics (Definition~\ref{def:el-curves}) and
as a minimum over embeddings of annuli (Proposition~\ref{prop:el-total-el}).

When $\SF[f] \ge 1$, Proposition~\ref{prop:SF-large-cover} serves this role.
When $\SF[f] < 1$, there are many different conformal embeddings
$R \hookrightarrow S$ in the homotopy class~$[f]$. The space of such
conformal embeddings is path-connected \cite{FB18:Couch}. One could
attempt to find a canonical embedding by, for instance, gluing annuli
to the boundary components of~$R$ \cite{EM78:ConfEmbeddings}. But this
embedding seems ill-suited to give tight bounds on $\SF[f]$ or
$\wt\SF[f]$. Ideally one would want a notion of ``map with
quasi-conformal constant less than one'', but that is nonsensical.

Instead, it seems likely we need to consider
some sort of ``smeared'' maps: maps from $R$ to
probability distributions on~$S$.

\begin{problem}\label{prob:SF-smeared}
 Find an energy of smeared maps $g \co R \to \Meas(S)$
  whose minimum value is $\wt\SF[f]$.
\end{problem}

As an example of what we mean, we give one way to get an explicit upper
bound on $\wt\SF[f]$.

\begin{definition}
  A homotopy class of topological embeddings $[f] \co R
  \hookrightarrow S$ between
  Riemann surfaces is \emph{conformally loose} if, for all $y \in \overline{S}$,
  there is a conformal embedding $g \in [f]$ so that
  $y \notin \overline{f(R)}$.

Since $\overline{S}$ is compact, if $[f]\co R \to S$ is conformally
loose we
can find finitely many conformal embeddings
$f_i \in [f], i = 1,\dots,n$ so that
\begin{equation}
  \bigcap_{i=1}^n \overline{f_i(R)} = \emptyset.
  \label{eq:finite-loose}
\end{equation}
In this case, we say that $[f]$ is \emph{$n$-loose}.
\end{definition}

\begin{proposition}\label{prop:sf-loose-bound}
  If $[f]\co R \hookrightarrow S$ is $n$-loose, then
  $\wt\SF[f] \le 1-1/n$.
\end{proposition}

\begin{proof}
  If $f$ is $n$-loose, then all covers are also $n$-loose. So
  it suffices to prove that $\SF[f] \le 1-1/n$.

  Let $(f_i)_{i=1}^n$ be the $n$ different embeddings from
  Equation~\eqref{eq:finite-loose}. For a simple multi-curve
  $C\in\Curves^+(R)$, let $q = q_{f(C)} \in
  \Quad^+(S)$ be the quadratic differential corresponding to
  $f(C)$ from Theorem~\ref{thm:quad-diff-height}. For at least
  one~$i$, we will have
  \[\frac{A_q(f_i(R))}{A_q(S)} \le 1-1/n\]
  by Lemma~\ref{lem:empty-intersect} below.
  Then the argument from case \eqref{item:strict-strict} $\Rightarrow$
  \eqref{item:strict-sf} of the proof of Theorem~\ref{thm:strict-emb}
  shows that $\EL_R[C] \le (1-1/n)\EL_S[f(C)]$, as desired.
\end{proof}

\begin{lemma}\label{lem:empty-intersect}
   If $A_1, \dots, A_n \subset X$ are $n$ subsets of a measure
   space~$X$ so that $\bigcap_{i=1}^n A_i = \emptyset$, then for at
   least one~$i$ we must have $\mu(A_i) \le (1-1/n) \mu(X)$.
\end{lemma}

\begin{proof}
  This follows from the continuous pigeonhole principle.
\end{proof}

In the language of Problem~\ref{prob:SF-smeared}, if $[f]$ is $n$-loose,
then the averaged map
\[
g(x) = \frac{1}{n} \sum_{i=1}^n f_i(x)
\]
is a smeared map from $R$ to~$S$. Likewise, if $\wt{f} \co \wt{R} \to
\wt{S}$ is $n$-loose where $q \co \wt{R} \to R$ is a finite cover of
degree~$k$, then the averaged map
\[
g(x) = \frac{1}{nk}\sum_{q(\wt{x}) = x} \,\sum_{i=1}^n \wt{f}_i(\wt{x})
\]
is a smeared map from $R$ to~$S$.

\begin{conjecture}\label{conj:loose}
  If $f \co R \to S$ is a strict conformal embedding of Riemann
  surfaces where $S$ has no punctures, there is some finite
  cover~$\wt f$
  of~$f$ that is conformally loose.
\end{conjecture}

If $[f]$ maps a puncture~$x$ of~$R$ to a puncture~$y$ of~$S$, a
neighborhood of~$y$ is in the image of every map in~$[f]$, so
$[f]$ can never be conformally loose. In this case we could
pass to a branched double cover as in the proof of
Proposition~\ref{prop:punct-disk-area}.

\begin{remark}
  In Problems~\ref{prob:inf-SF} and~\ref{prob:SF-smeared}, it may be
  that $\wt\SF[f]$ is not
  the most natural quantity to consider; there may be a more natural
  quantity that bounds $\wt\SF[f]$ from above and is less than one
  when $\wt\SF[f]$ is less than one.
\end{remark}


\bibliographystyle{hamsalpha}
\bibliography{conformal,topo,curves,drafts,dylan}

\end{document}